\documentclass[10pt]{amsart}
\usepackage{amsmath, amssymb}
\usepackage{mathrsfs}
\newcommand{\no}[1]{#1}
\renewcommand{\no}[1]{}
\no{\usepackage{times}\usepackage[subscriptcorrection, slantedGreek, nofontinfo]{mtpro}
\renewcommand{\Delta}{\upDelta}}
\usepackage{color}
\usepackage{enumerate,comment}

\date{\today}
\newtheorem{theorem}{Theorem}[section]
\newtheorem{proposition}{Proposition}[section]
\newtheorem{lemma}{Lemma}[section]

\newtheorem{corollary}{Corollary}[section]

\theoremstyle{remark}
\newtheorem{remark}{Remark}[section]

\newcommand{\bel}{\begin{equation} \label}
\newcommand{\ee}{\end{equation}}

\newcommand{\dd}{\mathrm d}

\newcommand{\R}{{\mathbb R}}

\newcommand{\U}{{\mathcal U}}

\newcommand{\Z}{{\mathbb Z}}

\def\beq{\begin{equation}}
\def\eeq{\end{equation}}
\newcommand{\bea}{\begin{eqnarray}}
\newcommand{\eea}{\end{eqnarray}}
\newcommand{\beas}{\begin{eqnarray*}}
\newcommand{\eeas}{\end{eqnarray*}}

\newcommand{\para}[1]{\left(#1\right)}
\newcommand{\cro}[1]{\left[#1\right]}
\newcommand{\norm}[1]{\left\Vert#1\right\Vert}
\newcommand{\abs}[1]{\left\vert#1\right\vert}
\numberwithin{equation}{section}

\title[Stability estimate for
the aligned magnetic field in a periodic quantum waveguide]{Stability estimate for
the aligned
magnetic field in a periodic quantum waveguide from Dirichlet-to-Neumann map }

\author[Youssef Mejri]{Youssef Mejri}
\address{CPT, UMR CNRS 7332, Aix Marseille Universit\'e, 13288 Marseille, France, D\'ep. Des Math\'ematiques, Facult\'e Des Sciences De Bizerte, 7021 Jarzouna, Tunise, Laboratoire De Mod\'elisation Math\'ematique Et Num\'erique Dans Les Sciences De L'Ing\'enieur, ENIT BP 37, Le Belvedere 1002 Tunis}
\email{josef-bizert@hotmail.fr}

\date{}

\begin{document}

\begin{abstract} 
In this article, we study the boundary inverse problem of determining the aligned magnetic field
appearing in the magnetic Schr\"odinger equation in a periodic quantum cylindrical waveguide. Provided that the Dirichlet-to-Neumann map of the magnetic Schr\"odinger equation. We prove a H\"older stability estimate with respect to the Dirichlet-to-Neumann map, by means of the geometrical optics solutions of the magnetic Schr\"odinger equation.


\end{abstract}

\maketitle

\tableofcontents

\section{Introduction}
\label{introduction}
\setcounter{equation}{0}

\subsection{Statement of the problem and existing papers}
In the present paper we consider an infinite waveguide $\Omega=\R \times \Omega'$, where
$\Omega'$ is a bounded domain of $\R^2$ with $C^{2}$-boundary $\partial \Omega'$.
We assume without limiting the generality of the foregoing that $\Omega'$ contains the origin.
For shortness sake we write $x=(x_1,x')$ with $x'=(x_2,x_3)$ for every $x=(x_1,x_2,x_3) \in \Omega$. Given $T>0$, we consider the following initial boundary value problem (IBVP in short) for the   Schr\"odinger equation with a magnetic potential,
\begin{equation}
\label{(1.1)}
\left\{
\begin{array}{ll}
(i\partial _t +\Delta_A)u=0 & \mbox{in}\ Q=(0,T)\times \Omega,
\\
u(0,\cdot)=0 &\mbox{in}\ \Omega,
\\
u=g &\mbox{on}\ \Sigma =(0,T) \times \partial \Omega,
\end{array}
\right.
\end{equation}
where
$$
\Delta_A=\sum_{j=1}^3(\partial_j+ia_j)^2=\Delta+2iA\cdot\nabla+i  \mbox{div}(A)-|A|^2.
$$
Throughout the entire paper we assume that the magnetic potential $A=(a_j)_{1\leq j\leq 3}\in W^{3,\infty}(\Omega;\R^3)$ is $1$-periodic with respect to the infinite variable $x_1$:
\begin{equation}\label{(1.2)}
A(x_1+1,\cdot)=A(x_1,\cdot),\ x_1 \in \R.
\end{equation}
Let us consider the Dirichlet-to-Neumann map (DN map in short)
\bel{a1}
\Lambda_A(g)=\left( \partial _\nu+iA\cdot\nu \right)u,
\ee
where $\nu=\nu\left(x\right)$ denotes the unit outward normal to $\partial\Omega$ at $x$ and $u$ is the solution to (\ref{(1.1)}). In this paper we address the inverse problem of determining the magnetic potential $A$ from the knowledge of $\Lambda_A$.
First of all, let us observe that there is an obstruction to uniqueness. In fact as it was noted in \cite{Esk8}, the DN map is invariant under the gauge transformation of the magnetic potential: it ensues from the identities
\begin{equation}\label{(40)}
e^{-i\Psi}\Delta_{A}e^{i\Psi}=\Delta_{A+\nabla\Psi},\,\,\,\,\,\,e^{-i\Psi}\Lambda_{A}e^{i\Psi}=\Lambda_{A+\nabla\Psi},
\end{equation}
imply that $\Lambda_A=\Lambda_{A+\nabla\Psi}$ when $\Psi\in \mathcal{C}^1(\Omega)$ is such that $\Psi_{\mid {\partial\Omega}}=0$. Therefore, the magnetic potential $A$ cannot be uniquely determined by the DN map $\Lambda_A$. From a geometric view point the vector field $A$ defines
the connection given by the one form $\alpha_A=\sum_{j=1}^3 a_j dx_j,$ and the non-uniqueness manifested in $(\ref{(40)})$ says that the best we could hope to reconstruct from the DN map  $\Lambda_A$ is the 2-form called
the magnetic field $d\alpha_A$ given by
$$
d{\alpha_A}=\sum_{i,j=1}^3\displaystyle(\frac{\partial a_i}{\partial x_j}-\frac{\partial a_j}{\partial x_i})\, \mbox{d}x_j\wedge \mbox{d}x_i.
$$
For the  Schr\"odinger equation in a domain with obstacles, Eskin proved in \cite{Esk8} that the knowledge of the DN map determines uniquely the potential. The main ingredient in his proof is the construction of geometric optics solutions. In the absence of the magnetic potential, Bellassoued and Ferreira prove in \cite{BD} that the DN map determines uniquely the electric potential. The problem of stability in determining the time-independent electric potential in a Schr\"odinger equation from a single boundary measurement was studied by Baudouin and Puel in \cite{LP}. This result was improved by Mercardo, Osses and Rosier in \cite{MOR}. Their method is essentially based on an appropriate Carleman inequality. In these two papers, the main assumption is that the part of the boundary where the measurement is made must satisfy a geometric condition insuring observability (see Bardos, Lebeau and Rauch \cite{BLR}). Rakesh and Symes proved in \cite{RS.1} that the knowledge of the DN map determines uniquely the time-independent scalar potentials in a wave equation. This result was extended to time-dependent scalar potential in \cite{RS.2}.In \cite{BB}, Bellassoued and Benjoud prove that the DN map determines uniquely the magnetic field induced by a magnetic potential in a magnetic wave equation. In this work, the DN map gives on the whole boundary. The uniqueness by a local DN map is well solved (Belishev \cite{B}, Eskin \cite{Esk7}, \cite{Esk19}, Eskin-Ralston \cite{ER}, Katchlov, Kurylev and lassas \cite{KKL}). \\ In \cite{BC}, Bellassoued and Choulli consider the problem of determining the time-independent magnetic field of the dynamic Schr\"odinger equation from the knowledge of the DN map.\\
In all the above mentioned articles the Schr\"odinger equation is defined in a bounded domain. Actually is only a small number of mathematical papers dealing with inverse boundary value problems in an unbounded domain see \cite{K}, \cite{KLU} and \cite{SW}. In \cite{LU} the compactly supported coefficients appearing in an infinite slab were identified from the knowledge of partial DN data. In \cite{CKS} Choulli, Kian and Soccorsi prove logarithmic stability in the determination of the time-dependent scalar potential in a 1-periodic quantum cylindrical waveguide from the knowledge of the DN map. In the present paper we rather investigate the problem of determining the aligned magnetic field appearing in the magnetic Schr\"odinger equation in a periodic quantum cylindrical waveguide from the DN map.
\subsection{Main results}
In this subsection we state the main results of this article. We first need to define the trace operator $\tau$ by
$$
\tau w=w_{|\Sigma},\,\,\,\mbox{for}\,\, w\in C^{\infty}_0\para{[0,T]\times\R,C^\infty\para{\overline{\Omega'}}}.
$$
Recall that since $\Omega'$ is a domain of $\R^2$ with $C^{2}$-boundary, we can extend $\tau$ to a bounded operator from $H^2\para{0,T; H^2\para{\Omega}}$ into $L^2\para{\para{0,T}\times\R,H^{3/2}\para{\partial \Omega'}}$. Then the space $X_0=\tau{H^2\para{0,T; H^2\para{\Omega}}}$ endowed with the norm
$$
\parallel w\parallel_{X_0}=\inf\{\|W\|_{H^2\para{0,T; H^2\para{\Omega}}};\,W\in H^2\para{0,T; H^2\para{\Omega}}\,\,\mbox{such that}\,\tau W=w\},
$$
is Hilbertain. Moreover, as will be seen in Section $ \ref{sec-bo} $, the linear operator $\Lambda_A$ defined by \eqref{a1}, is bounded from $X_0$ to $L^2\para{\Sigma}$. Next, we introduce the following subset of $W^{1,\infty}\para{\Omega; \R^3}$
$$
\mathcal{A}=\{ A\in W^{1,\infty}\para{\Omega; \R^3}; \,\, \|A\|_{L^{\infty}\para{\Omega}}<\frac{\sqrt{2}-1}{C\para{\Omega'}^{1/2}}\},
$$
where ${C\para{\Omega'}}$ is the smallest Poincar\'e constant associated with $\Omega'$, i.e the smallest of those constants $C>0$ such that we have $$\displaystyle \parallel u \parallel_{L^2\para{\Omega}}^2\leq C\parallel \nabla u \parallel_{L^2\para{\Omega}}^2,\,\,u\in H_0^1\para{\Omega
}.$$
Last, putting
\bel{che}
\check{\Omega}=(0,1)\times \Omega',\ \check{Q}=(0,T)\times \check{\Omega},\ \check{\Sigma}=(0,T)\times (0,1)\times \partial \Omega',
\ee
we may now state the main result of this paper.
\begin{theorem}
\label{thm1}
Let $M>0$ and $A_1,\ A_2 \in W^{3,\infty}(\Omega;\R^3)\cap \mathcal{A}$ fulfill $\eqref{(1.2)}$ together with the
three following conditions:
\begin{equation}
\|A_i\|_{W^{3,\infty}(\check{\Omega})}\leq M,\ \ i=1,2,
\end{equation}
\bel{(1.3)}
A_1=A_2 \,\,\, \mbox{on}\ \para{0,1}\times \partial\Omega',
\ee
\bel{(1.4)}
\partial_jA_1=\partial_jA_2 \,\,\,\mbox{on}\  \para{0,1}\times \partial\Omega',\,\,j=1,2,3.
\ee
Then there exist a constants  $C>0$  depending only on $A_1$, $T$, $\Omega'$ and $M$, such that we have
$$
\left\| \frac{\partial a_2}{\partial x_3}-\frac{\partial a_3}{\partial x_2} \right\|_{L^2(\check{\Omega})}\leq C \| \Lambda_{A_2}-\Lambda_{A_1} \|^{
2/45}.
$$
\end{theorem}
Theorem \ref{thm1} follows from a result we shall make precise below, which is related to the following IBVP with quasi-periodic boundary conditions,
\begin{equation}
\label{(1.6)}
\left\{
\begin{array}{ll}
(i\partial_t +\Delta_ A) u=0 & \mbox{in}\ \check{Q}, \\
u(0,\cdot )=0 & \mbox{in}\ \check{\Omega}, \\
u=h & \mbox{on}\ \check{\Sigma},
\\
u(\cdot,1,\cdot)=e^{i\theta} u(\cdot,0,\cdot) & \mathrm{on}\ (0,T) \times \Omega',
\\
\partial_{x_1} u(\cdot,1,\cdot)=e^{i\theta} \partial_{x_1} u(\cdot,0,\cdot) & \mathrm{on}\ (0,T) \times \Omega',
\end{array}
\right.
\end{equation}
where $\theta$ is arbitrarily fixed in $[0,2\pi)$. To this purpose, for any subspace $\mathcal{O}=\para{0,1}\times\R^2$ or $\R^3$, we take
$$ H_{\theta}^1(\mathcal{O}) =\{ u \in H^1(\mathcal{O});\ u(1,\cdot)=e^{i \theta} u(0,\cdot)\,\, \mbox{in}\,\, \R^2 \},$$
and
$$ H_{\theta}^2(\mathcal{O}) =\{ u \in H^2(\mathcal{O});\ u(1,\cdot)=e^{i \theta} u(0,\cdot)\ \mbox{and}\ \partial_{x_1} u(1,\cdot)=e^{i \theta} \partial_{x_1} u(0,\cdot)\,\, \mbox{in}\,\,\R^2 \}.$$
We denote by $\check{\tau}$ the linear bounded operator from $H^2\para{0,T;H^2\para{\check{\Omega}}}$ into $L^2\para{\para{0,T}\times\para{0,1},H^{3/2}\para{\partial \Omega'}}$, such that
$$
\check{\tau} w=w_{|\check{\Sigma}}\,\,\mbox{for}\,\, w\in C^{\infty}_0\para{[0,T]\times\para{0,1},C^\infty\para{\overline{\Omega'}}}.
$$
and by $\check{X}_\theta$ the space $\check{X}_\theta=\check{\tau}\para{H^2\para{0,T;H_\theta^2\para{\check{\Omega}}}
}$. As will appear in Section $\ref{sec-FD},$ the operator
\bel{def-fibbdry}
\Lambda_{A,\theta} : h\in \check{X}_\theta   \longmapsto \left(
\partial_\nu +iA\cdot\nu\right)u \in L^2\para{\check{\Sigma}},
\ee
where $u$ is the solution to \eqref{(1.6)}, is bounded. The following result essentially claims that Theorem \ref{thm1} remains valid upon substituting $\Lambda_{A_j,\theta}$ for $A_j,\,\,\,j=1,2,\,\,\,$ for $\theta$ arbitrary in $[0,2\pi).$
\begin{theorem}
\label{thmm}
Let $A_1$ and $A_2$ be the same as in Theorem \ref{thm1}. Then we may find a
constants $C>0$ depending only $A_1,$ $T,$ $M$ and $\Omega'$, such that we have
$$
\left\| \frac{\partial a_2}{\partial x_3}-\frac{\partial a_3}{\partial x_2} \right\|_{L^2(\check{\Omega})}\leq C \| \Lambda_{A_2,\theta}-\Lambda_{A_1,\theta} \|^{
2/45},
$$
for any $\theta \in [0,2 \pi)$.
\end{theorem}
\subsection{Outline}
The remainder of this paper is organized as follows. In section \ref{sec-bo} we analyze the direct problem associated with the IBVP (\ref{(1.1)}) and we prove that the boundary operator $\Lambda_A$ is bounded. In section 3 we use the Floquet-Bloch-Gel'fand transform to decompose the IBVP \eqref{(1.1)} into a collection of problems \eqref{(1.6)}. In section 4, we construct the geometric optics solutions to the above mentioned quasi-periodic boundary value problem. These solutions constitute the main ingredient in the proof of the above stability. Finally, the proof of Theorems \ref{thm1} and \ref{thmm} is detailed in section 5.


\section{Analysis of the direct problem}
\label{sec-bo}
In this section we examine the direct problem associated to $\para{\ref{(1.1)}}$.
\subsection{The magnetic  Schr\"odinger operator}
Let $\Omega$ be as before. We introduce the Dirichlet Laplacian $L_0=-\Delta$  in $L^2\para{\Omega}$
which is the linear self-adjoint operator in $L^2\para{\Omega}$ generated by the closed quadratic form
$$
\ell_0\para{u}=\int_{\Omega}\abs{\nabla u\para{x}}^2dx,\,\,\,u\in \mathcal{D}\para{\ell_0}=H_0^1\para{\Omega}.
$$
Its domain is $\mathcal{D}\para{L_0}=H_0^1\para{\Omega}\cap H^2\para{\Omega},$ by \cite{CKS}[Lemma 2.2].
We perturb $\ell_0$ with the magnetic potential $A\in W^{1,\infty}\para{\Omega,\R^3}$ and we consider
$$
\ell_A\para{u}=\int_{\Omega}\abs{\nabla_A u\para{x}}^2dx,\,\,\,u\in  \mathcal{D}\para{\ell_A}=H_0^1\para{\Omega},
$$
where $\nabla_A=\nabla+iA$.\\
Let us establish now that the metric induced by $\ell_0$ is equivalent to the one induced by $\ell_A$ provided $A$ is sufficiently small.
\begin{lemma}\label{lm.1}
Let $L_A$ be the self-adjoint operator in $L^2\para{\Omega}$ generated by the closed quadratic form $\ell_A$. Then the operator $L_A$ acts as $\para{-\Delta_A}$ on its domain $\mathcal{D}\para{L_A}=\mathcal{D}\para{-\Delta}=H_0^1\para{\Omega}\cap H^2\para{\Omega}$. Moreover, if $A\in\mathcal{A},$ there exists a positive constant $C$ such that we have
\begin{equation}\label{91}
\|\nabla u\|_{L^2\para{\Omega}}\leq C \|\nabla_A u\|_{L^2\para{\Omega}},\,\mbox{for all}\,u\in H^1_0\para{\Omega}.
\end{equation}
\end{lemma}
\begin{proof} We introduce the form
$$
\beta(u)= 2\mbox{Im}\para{\nabla u,Au}_{L^2\para{\Omega}}+\|Au\|^2_{L^2\para{\Omega}},\,\,\,u\in H^1_0\para{\Omega}.
$$
So we have
$$
\ell_A(u)=\ell_0(u)+ \beta(u),
$$
and for every $\epsilon\in(0,1)$
$$
\begin{array}{lll}
\displaystyle|\beta(u)|&\leq &\epsilon \|\nabla u\|^2_{L^2\para{\Omega}}+\para{1+\epsilon^{-1} }\|Au\|^2_{L^2\para{\Omega}}\\&\leq &\displaystyle \epsilon \ell_0(u)+\para{1+\epsilon^{-1} }\|A\|^2_{L^\infty\para{\Omega}}\|u\|^2_{L^2\para{\Omega}}.
\end{array}
$$
Thus the operator $2i A\cdot\nabla+i \mbox{div} (A)-|A|^2$ associated with $\beta$ in $L^2\para{\Omega}$, with domain $\mathcal{D}\para{\beta}= H_0^1\para{\Omega}$, is relatively bounded w.r.t $\para{-\Delta}$, with relative bound $\epsilon < 1$. By \cite{RS3}[Theorem X.17] the operator $-\Delta_A=-\Delta-2i A\cdot\nabla-i \mbox{div} (A)+|A|^2 $ is  self-adjoint in $L^2\para{\Omega}$ and $\mathcal{D}\para{-\Delta_A}=D\para{-\Delta}=H_0^1\para{\Omega}\cap H^2\para{\Omega}.$ Moreover we have
$$
\ell_A(u)\geq\para{1-\epsilon} \ell_0(u)-\para{1+\epsilon^{-1} }\|A\|^2_{L^\infty\para{\Omega}}\|u\|^2_{L^2\para{\Omega}}.
$$
As
$$
\|u\|^2_{L^2\para{\Omega}}\leq C\para{\Omega'} \ell_0(u),\,\,u\in H^1_0\para{\Omega},
$$
by Poincar\'e inequality, this implies that
\begin{equation}\label{01}
\ell_A(u)\geq\para{\para{1-\epsilon}-C\para{\Omega'}\para{1+\epsilon^{-1} }\|A\|^2_{L^\infty\para{\Omega}}}\ell_0(u).
\end{equation}
Now, we choose
$$
\epsilon=\epsilon_0=C\para{\Omega'}^{1/2} \|A\|_{L^\infty\para{\Omega}}\,\in \para{0,1}
$$
since $A\in\mathcal{A}$. Then we have
$$
\ell_A(u)\geq\para{2-\para{1+C\para{\Omega'}^{1/2} \|A\|_{L^\infty\para{\Omega}}}^2}\ell_0(u),\,\,u\in H^1_0\para{\Omega}.
$$
Thus, taking $C=\para{2-\para{1+C\para{\Omega'}^{1/2} \|A\|_{L^\infty\para{\Omega}}}^2}>0$, we obtain directly $\para{\ref{91}}$.
\end{proof}
\subsection{Existence and uniqueness results}
In this subsection we study the direct problem associated with the IBVP $\para{\ref{(1.1)}}$. To this end we consider the abstract evolution problem
\begin{equation}\label{l1}
\left\{
  \begin{array}{lll}
  &iv'\para{t}+\Delta_Av(t)=f,\,\,\,\,\,\,\,\,\,\,\,\,t\in(0,T),
   \\
     &\hbox{$ v\para{0}=v_0 $},
  \end{array}
\right.
\end{equation}
with initial data $v_0$ and source $f$. We shall derive some properties of the solution to \eqref{l1} with the aid of the following technical result, which is borrowed from \cite{CKS}$[\mbox{Lemma}\,2.1]$.
\begin{lemma}\label{lm1}
Let $X$ be a Banach space, $U$ be a m-dissipative operator in $X$ with dense domain $\mathcal{D}\para{U}$. Then for all $v_0\in \mathcal{D}\para{U}$ and $f\in\mathcal{C}\para{[0,T];X}\cap L^1\para{[0,T];\mathcal{D}\para{U}}$ $\para{\mbox{resp.}\,\,f\in W^{1,1}\para{0,T;X}}$ there is a unique solution $v\in Z_0:=\mathcal{C}\para{[0,T];\mathcal{D}\para{U}}\cap \mathcal{C}^1\para{[0,T];X} $ to the following Cauchy problem
\begin{equation}\label{l2}
\left\{
  \begin{array}{lll}
  &v'\para{t}=U v\para{t}+ f\para{t},
   \\
     &\hbox{$ v\para{0}=v_0 $},
  \end{array}
\right.
\end{equation}
such that
\begin{equation}
\|v\|_{Z_0}=\|v\|_{\mathcal{C}\para{[0,T];\mathcal{D}\para{U}}}+\|v\|_{\mathcal{C}^1\para{[0,T];X}} \leq C \|f\|_\ast.
\end{equation}
Here $C$ is some positive constant depending only on $T$ and $ \|f\|_\ast$ stands for the norm $ \|f\|_{\mathcal{C}\para{[0,T];X}\cap L^1\para{0,T;\mathcal{D}\para{U}}}$ $\para{\mbox{resp}.\,\,\|f\|_{W^{1,1};\para{0,T;X}}}.$
\end{lemma}
We are now in position to establish the following existence and uniqueness result.
\begin{proposition}
\label{pr-exiuni}
Let $A\in W^{1,\infty}\para{\Omega; \R^3}$. Then for all $ f\in W^{1,1}\para{0,T,L^2\para{\Omega}}$ there exists a unique solution $v\in\mathcal{C}\para{\cro{0,T};H_0^1\para{\Omega}\cap H^2\para{\Omega}}\cap\mathcal{C}^1\para{\cro{0,T};L^2\para{\Omega}}$ to
\begin{equation}\label{28}
\left\{
  \begin{array}{lll}
  &\para{i\partial_t+\Delta_A}v=f\,\,\,\,\,\,\,\,\,\,\,\,\,\,\,\,\,\,\,\,\,\,\,\,\,\,\,\,\,\,\,\,\,\,\,\,\,\,\,\,\,\,\,\,\,\,\,\,\,\,\,\,\,\,\,\,\,\,\mbox{in}\,\,Q,
   \\
     &\hbox{$ v\para{0,.}=0\,\,\,\,\,\,\,\,\,\,\,\,\,\,\,\,\,\,\,\,\,\,\,\,\,\,\,\,\,\,\,\,\,\,\,\,\,\,\,\,\,\,\,\,\,\,\,\,\,\,\,\,\,\,\,\,\,\,\,\,\,\,\,\,\,\,\,\,\,\,\,\,\,\,\mbox{in}\,\,\Omega $},\\ & v=0\,\,\,\,\,\,\,\,\,\,\,\,\,\,\,\,\,\,\,\,\,\,\,\,\,\,\,\,\,\,\,\,\,\,\,\,\,\,\,\,\,\,\,\,\,\,\,\,\,\,\,\,\,\,\,\,\,\,\,\,\,\,\,\,\,\,\,\,\,\,\,\,\,\,\,\,\,\,\,\,\,\,\,\,\,\,\,\mbox{on}\,\,\Sigma.
  \end{array}
\right.
\end{equation}
Moreover $v$ fulfills
$$
\norm{v}_{\mathcal{C}\para{\cro{0,T};H^2\para{\Omega}}}+ \|v\|_{\mathcal{C}^1\para{\cro{0,T};L^2\para{\Omega}}}\leq C \|f\|_{W^{1,1}\para{0,T;L^2\para{\Omega}}},
$$
for some constant $C>0$ depending only on $\Omega',\,T$ and $\norm{A}_{W^{1,\infty}\para{\Omega}}$.
\end{proposition}
\begin{proof}
According to Lemma \ref{lm.1} we have apply Lemma \ref{lm1} to $U=-\Delta_A$ and $v_0=0.$ We find the existence of a unique solution $v\in\mathcal{C}\para{\cro{0,T};H_0^1\para{\Omega}\cap H^2\para{\Omega}}\cap\mathcal{C}^1\para{\cro{0,T};L^2\para{\Omega}}$ to $\para{\ref{28}}$.
\end{proof}
\begin{corollary}\label{cor2.1}
Let $A$ be the same as in Proposition $\ref{pr-exiuni}$. Then for every $g\in X_0$, the IBVP $\para{\ref{(1.1)}}$ admits a unique\footnote{The coming proof actually establishes that this solution belongs to $\mathcal{C}\para{\cro{0,T};H_0^1\para{\Omega}\cap H^2\para{\Omega}}\cap\mathcal{C}^1\para{\cro{0,T};L^2\para{\Omega}}$} solution
$$
\mathfrak{s}\para{g}\in Z=L^2\para{0,T; H^2\para{\Omega}}\cap H^1\para{0,T;L^2\para{\Omega}}.
$$
Moreover, we have
\begin{equation}\label{cor.1}
\norm{\mathfrak{s}\para{g}}_Z \leq C \norm{g}_{X_0},
\end{equation}
for some constant $C$ depending only on $\Omega',\,T$ and $\norm{A}_{W^{1,\infty}\para{\Omega}}.$
\end{corollary}
\begin{proof}
Choose $G\in H^2\para{0,T; H^2\para{\Omega}}$ obeying $G_{\mid \Sigma}=g,\,G\para{0,.}=0$ and $\|G\|_{H^2\para{0,T; H^2\para{\Omega}}}=\|g\|_{X_0}$. Then $u$ is solution to $\para{\ref{(1.1)}}$ if and only if $u-G$ is solution to $\para{\ref{28}}$ with $f=-\para{i\partial_t+\Delta_A }G$. Therefore the result follows from this and Proposition $\ref{pr-exiuni}$.
\end{proof}
Armed with Corollary \ref{cor2.1} we are now in position to define the Dirichlet-to-Neumann map $\Lambda_A$. First for all, we need to introduce the trace operator $\tau_1$, defined as the linear bounded operator from $L^2\para{\para{0,T}\times\R;H^2\para{\Omega'}}\cap H^1\para{0,T;L^2\para{\Omega}}$ into $L^2\para{\Sigma}$, which coincides with the mapping
$$
\omega \longmapsto \para{\partial_\nu + i A\cdot\nu}\omega_{\mid \Sigma}\,\,\,\,\mbox{for}\,\omega\in\mathcal{C}_0^\infty\para{[0,T]\times\R;\mathcal{C}^\infty\para{\overline{\Omega'}}}.
$$
Evidently, we have
$$
\norm{\tau_1 \mathfrak{s}\para{g}}_{L^2\para{\Sigma}}\leq C \norm{\mathfrak{s}\para{g}}_{Z}\leq C\norm{g}_{X_0},
$$
by \eqref{cor.1}, where $C>0$ denotes a generic constant that does no depend on $g$. Hence the linear operator
\begin{equation}
\Lambda_A=\tau_1\circ \mathfrak{s},
\end{equation}
is bounded from $X_0$ into $L^2\para{\Sigma}$ with $\norm{\Lambda_A}=\norm{\Lambda_A}_{\mathcal{L}\para{X_0,L^2\para{\Sigma}}}\leq C,$ the constant $C$ depending on $\norm{A}_{W^{1,\infty}\para{\Omega}}.$
\section{Fiber decomposition}\label{sec-FD}
In this section we decompose the Cauchy problem \eqref{(1.1)} into a collection of IBVP with quasi-periodic boundary conditions of the form \eqref{(1.6)}, with the aid of the partial Floquet-Bloch-Gel'fand transform (abbreviated to FBG in the sequel).
We start by recalling the definition of this transform.
\subsection{Partial FBG transform}
\label{sec-FBG}
Let $f$ be a function on $ C_0^\infty(Q)$. We define the partial FBG transform with respect to $x_1$ of $f$ by
\bel{c1}
\check{f}_{\theta}(x_1,y)=(\U f)_\theta (t,x)=\sum_{k=-\infty}^{+\infty}e^{-ik\theta} f(x_1+k,y),\ (x_1,y) \in \R \times \Omega,\  \theta \in [0,2 \pi).
\ee
With reference to \cite[Section XIII.16]{RS4}, $\U$ extends to a unitary operator, still denoted by $\U$, from $L^2(Q)$ onto the Hilbert space
$\int_{(0, 2 \pi)}^\oplus L^2(\check{\Omega})\dd \theta/2 \pi=L^2 \left( (0,2 \pi) \dd \theta/2 \pi, L^2( (0,T) \times \check{\Omega}) \right).$ The main benefit of using the partial FGB transform when dealing with a IBVP with periodic coefficients such as \eqref{(1.1)} can be understood from the following result, which is borrowed from  \cite{CKS}[Proposition 6.1].
\begin{proposition}
\label{pr-equiv}
Let $A \in \mathcal{A}$ fulfill \eqref{(1.2)} and let $g \in X_0$. Then $u$ is the solution $\mathfrak{s}(g) \in Z$ to \eqref{(1.1)} defined in Corollary \ref{cor2.1} if and only if
each $\check{u}_\theta=(\U u)_\theta \in L^2(0,T;H_\theta^2(\check{\Omega})) \cap H^1([0,T];L^2(\check{\Omega}))$, $\theta \in [0,2 \pi)$, is solution to \eqref{(1.6)} with $h=\check{g}_\theta$.
\end{proposition}

We now examine the direct problem associated to the fibered  IBVP \eqref{(1.6)} and we define the fibered boundary operators $\Lambda_{A,\theta}$, where the real number $\theta$ is arbitrary in $[0, 2 \pi)$.
\subsection{Analysis of the direct fibered problem}
Let us prove the following existence and uniqueness result.
\begin{lemma}
\label{lm-c}
Let $A\in \mathcal{A}$ fulfill $\para{\ref{(1.2)}}$. Then for every $f\in W^{1,1}\para{0,T;L^2\para{\check{\Omega}}}$ there is a unique solution $w\in \check{Z}_\theta=L^2(0,T;H_\theta^2(\check{\Omega})) \cap H^1(0,T;L^2(\check{\Omega}))$ to the following initial boundary value problem
\begin{equation}\label{eq:}
\left\{
  \begin{array}{lll}
  &\para{i\partial_t+\Delta_A}w=f\,\,\,\,\,\,\,\,\,\,\,\,\,\,\,\,\,\,\,\,\,\,\,\,\,\,\,\,\,\,\,\,\,\,\,\,\,\,\,\,\,\,\,\,\,\,\,\,\,\,\,\,\,\,\mbox{in}\,\,\check{Q},
   \\
     &\hbox{$ w\para{0,.}=0\,\,\,\,\,\,\,\,\,\,\,\,\,\,\,\,\,\,\,\,\,\,\,\,\,\,\,\,\,\,\,\,\,\,\,\,\,\,\,\,\,\,\,\,\,\,\,\,\,\,\,\,\,\,\,\,\,\,\,\,\,\,\,\,\,\,\,\,\,\,\,\,\mbox{in}\,\,\check{\Omega} $},\\ & w=0\,\,\,\,\,\,\,\,\,\,\,\,\,\,\,\,\,\,\,\,\,\,\,\,\,\,\,\,\,\,\,\,\,\,\,\,\,\,\,\,\,\,\,\,\,\,\,\,\,\,\,\,\,\,\,\,\,\,\,\,\,\,\,\,\,\,\,\,\,\,\,\,\,\,\,\,\,\,\,\,\,\,\,\,\,\mbox{on}\,\,\check{\Sigma},\\&w\para{.,1,.}=e^{i\theta}w\para{.,0,.}\,\,\,\,\,\,\,\,\,\,\,\,\,\,\,\,\,\,\,\,\,\,\,\,\,\,\,\,\,\,\,\,\,\,\,\,\,\,\,\,\,\mbox{in}\,\,\para{0,T}\times\Omega',\\
  &\partial_{x_1}w\para{.,1,.}=e^{i\theta}\partial_{x_1}w\para{.,0,.}\,\,\,\,\,\,\,\,\,\,\,\,\,\,\,\,\,\,\,\,\,\,\,\,\mbox{in}\,\,\para{0,T}\times\Omega'.
  \end{array}
\right.
\end{equation}
Moreover, we may find a constant $C=C\para{T,\Omega',\norm{A}_{W^{1,\infty}\para{\Omega}}}>0$ such that the estimates
\begin{equation}\label{.1}
 \displaystyle\norm{w}_{L^2\para{\check{Q}}}\leq \norm{f}_{L^1\para{0,T;L^2\para{\check{\Omega}}}},
 \end{equation}
and
\begin{equation}\label{.2}
 \norm{\nabla w}_{L^2\para{\check{Q}}}\leq C\para{\epsilon^{-1}\norm{f\para{t}}_{L^1\para{0,T;L^2\para{\check{\Omega}}}}+2\epsilon\norm{f'\para{t}}_{L^1\para{0,T;L^2\para{\check{\Omega}}}}},
\end{equation}
hold for every  $0<\epsilon\leq 1$ and $\theta \in [0,2 \pi)$.
\end{lemma}
\begin{proof}
Let $L_{\theta}$ be the self-adjoint operator in $L^2(\check{\Omega})$ generated by the closed quadratic form
$$ \ell_{\theta}(u)=\int_{\check{\Omega}} | \nabla_A u(x) |^2 \dd x,\,\, u \in \mathcal{D}(\ell_{\theta}) = L^2(0,1;H_0^1(\Omega')) \cap H_{\theta}^1(0,1;L^2(\Omega')), $$
in such a way that $L_{\theta}$ acts as $(-\Delta_A)$ on its domain $\mathcal{D}(L_{\theta}) = L^2(0,1;H_0^1(\Omega'))\cap H_{\theta}^2(\check{\Omega})$ according to \cite{CKS}[Lemma 3.1]. Therefore, applying Lemma \ref{lm1} with $U=L_\theta$ and $X=L^2\para{\check{\Omega}}$ we get for every $f\in W^{1,1}(0,T;L^2(\check{\Omega}))$ that there is a unique solution
$w \in L^2(0,T;L^2(0,1;H_0^1(\Omega')) \cap H_{\theta}^2(\check{\Omega})) \cap H^1(0,T;L^2(\check{\Omega}))$ to the IBVP \eqref{eq:}. Moreover estimates  \eqref{.1} and \eqref{.2} follow readily from \cite{BC}  [Lemma 3.2]. This completes the proof of the lemma.
\end{proof}
\begin{remark}\label{R2}
Let $h\in \check{X}_\theta$. From the definition of  $\check{X}_\theta$ we may find $W \in H^2(0,T;H_{\theta}^2(\check{\Omega}))$ such that $\check{\tau} W =h$ and $W\para{0,.}=0$. Thus, taking $f=\para{i\partial_t+\Delta_A} W$, it is obvious that $W-w$ is solution to \eqref{(1.6)} if and only if $w$ is solution to \eqref{eq:}. This implies that for every $h\in \check{X}_\theta$ there exists a unique solution $\mathfrak{s}_\theta(h) \in \check{Z}_\theta$ to the initial boundary value problem \eqref{(1.6)}. Moreover there exist $C=C\para{\Omega',\,T,\,\norm{A}_{W^{1,\infty}}}>0$ such that,
\begin{equation}\label{b.1}
\norm{\mathfrak{s}_\theta\para{h}}_{\check{Z}_\theta} \leq C \norm{h}_{\check{X}_\theta}.
\end{equation}
\end{remark}
\begin{remark}\label{R1}
We have a similar result as in Remark \ref{R2} by replacing the initial condition $u\para{0,.}=0$ in  \eqref{(1.6)}  by the final condition $u\para{T,.}=0.$ To see this we take $v\para{t,x}=u\para{T-t,x},\,\,\para{t,x}\in \check{Q}$, notice that $u$ is solution to  the boundary value problem (BVP in short)
\begin{equation}
\left\{
\begin{array}{ll}
(i\partial_t +\Delta_ A) u=0 & \mbox{in}\ \check{Q}, \\
u(T,\cdot )=0 & \mbox{in}\ \check{\Omega}, \\
u=h & \mbox{on}\ \check{\Sigma},
\\
u(\cdot,1,\cdot)=e^{i\theta} u(\cdot,0,\cdot) & \mathrm{on}\ (0,T) \times \Omega',
\\
\partial_{x_1} u(\cdot,1,\cdot)=e^{i\theta} \partial_{x_1} u(\cdot,0,\cdot) & \mathrm{on}\ (0,T) \times \Omega',
\end{array}
\right.
\end{equation}
if and only if $v$ is solution to the system
\begin{equation}
\left\{
\begin{array}{ll}
(-i\partial_t +\Delta_ A) v=0 & \mbox{in}\ \check{Q}, \\
v(0,\cdot )=0 & \mbox{in}\ \check{\Omega}, \\
v=h & \mbox{on}\ \check{\Sigma},
\\
v(\cdot,1,\cdot)=e^{i\theta} v(\cdot,0,\cdot) & \mathrm{on}\ (0,T) \times \Omega',
\\
\partial_{x_1} v(\cdot,1,\cdot)=e^{i\theta} \partial_{x_1} v(\cdot,0,\cdot) & \mathrm{on}\ (0,T) \times \Omega',
\end{array}
\right.
\end{equation}
and apply Remark \ref{R2}.
\end{remark}
Having seen this, we may now define the fibered boundary operator from \eqref{def-fibbdry}
\begin{equation}\label{lm3.1}
\Lambda_{A,\theta} : h \mapsto  \left(
\partial_\nu +iA\cdot\nu\right)u
\end{equation}
where $u$ is the solution to \eqref{(1.6)}. Let $\check{\tau}_1$ be the linear bounded operator from $L^2\para{0,T;H^2\para{\check{\Omega}}}\cap \,H^1\para{0,T;L^2\para{\check{\Omega}}}$ to $L^2\para{\check{\Sigma}}$, obeying
$$
\omega \longmapsto \para{\partial_\nu + i A\cdot\nu}\omega_{\mid \check{\Sigma}}\,\,\,\,\mbox{for}\,\omega\in\mathcal{C}_0^\infty\para{[0,T]\times\para{0,1},\mathcal{C}^\infty\para{\overline{\Omega'}}}.
$$
Then for every $h\in \check{X}_\theta,\,\theta\in [0,2\pi)$, we have
$$
\norm{ \check{\tau}_1 \mathfrak{s}_\theta\para{h}}_{L^2\para{\check{\Sigma}}}\leq C \norm{\mathfrak{s}_\theta\para{h}}_{\check{Z}_\theta}\leq C\norm{ h}_{\check{X}_\theta},
$$
by \eqref{b.1}, hence the linear operator
\begin{equation}
\Lambda_{A,\theta}=\check{\tau}_1\circ \mathfrak{s}_\theta,\,\,\,\,\theta\in[0,2\pi),
\end{equation}
is bounded from $\check{X}_\theta$ into $L^2\para{\check{\Sigma}}$ with $\norm{\Lambda_A}=\norm{\Lambda_A}_{\mathcal{L} \para{\check{X}_\theta,L^2\para{\check{\Sigma}}}}\leq C,$ where $C>0$ is a constant  depending on $\norm{A}_{W^{1,\infty}\para{\Omega}}.$ Further,
In view of \cite{CKS}[Proposition 6.2] the operator $\Lambda_A$ is equivalent to the direct integral of $\{\Lambda_{A,\theta},\,\theta\in [0,2\pi)\}$ as claimed below.
\begin{proposition}\label{dec-boun}
Let $A$ be the same as in Proposition \ref{pr-equiv}. Then we have
$$
\mathcal{U}\Lambda_A \mathcal{U}^{-1}=\int^\oplus_{(0,2\pi)} \Lambda_{A,\theta} \frac{\dd\theta}{2\pi}.
$$
Moreover in light of  \cite[Chap. II, \S 2, Proposition 2]{Di}, this entails that
\bel{es}
\norm{\Lambda_{A}} = \sup_{\theta \in (0,2 \pi)} \norm{\Lambda_{A,\theta}} .
\ee
\end{proposition}
We turn now to examining the inverse problem of determining $ \frac{\partial a_2}{\partial x_3}-\frac{\partial a_3}{\partial x_2} $ from the knowledge of $\Lambda_{A,\theta},$ where the real number $\theta$ is arbitrary in $[0,2\pi)$.
\section{Geometric optics solutions}
\label{sec-ogs}
 In this section we define geometric optics solutions for the magnetic Schr\"odinger equation which are useful for the derivation of main result. More precisely, we build geometric optics solutions to the system
\begin{equation}
\label{(4.1)}
\left\{
  \begin{array}{lll}
    &\para{ i\partial_t +\Delta_A}u=0\,\,\,\,\,\,\,\,\,\,\,\,\,\,\,\,\,\,\,\,\,\,\,\,\,\,\,\,\,\,\,\,\,\,\,\,\,\,\,\,\,\,\,\,\,\,\,\mbox{in}\,\,\check{\Omega},
    \\&u(.,1,.)=e^{i\theta}u(.,0,.)\,\,\,\,\,\,\,\,\,\,\,\,\,\,\,\,\,\,\,\,\,\,\,\,\,\,\,\,\,\,\,\,\,\,\,\,\,\mbox{in}\,\,(0,T)\times\Omega'
  ,\\&\partial_{ x_1}u(.,1,.)=e^{i\theta}\partial _{x_1}u(.,0,.)\,\,\,\,\,\,\,\,\,\,\,\,\,\,\,\,\,\,\,\,\mbox{in}\,\,(0,T)\times\Omega',
  \end{array}
\right.
\end{equation}
where $\theta$ is arbitrarily fixed in  $[0,2 \pi)$.\\Let $w=w\para{t,x}\in \mathcal{C}^1\para{\cro{0,T};L^2\para{\check{\Omega}}}\cap \mathcal{C}\para{\cro{0,T};H^2\para{\check{\Omega}}}$ satisfy the conditions
\begin{equation}\label{87}
\left\{
  \begin{array}{lll}
    &w\para{0,.}=0\,\,\,\,\,\,\,\,\,\,\,\,\,\,\,\,\,\,\,\,\,\,\,\,\,\,\,\,\,\,\,\,\,\,\,\,\,\,\,\,\,\,\,\,\,\,\,\,\,\,\,\,\,\,\,\,\,\,\,\,\,\,\,\,\,\,\,\,\,\,\,\,\,\,\,\,\mbox{in}\,\,\check{\Omega},
    \\&w=0\,\,\,\,\,\,\,\,\,\,\,\,\,\,\,\,\,\,\,\,\,\,\,\,\,\,\,\,\,\,\,\,\,\,\,\,\,\,\,\,\,\,\,\,\,\,\,\,\,\,\,\,\,\,\,\,\,\,\,\,\,\,\,\,\,\,\,\,\,\,\,\,\,\,\,\,\,\,\,\,\,\,\,\,\,\,\,\,\,\mbox{in}\,\,\check{\Sigma}
    \\&w\para{.,1,.}-e^{i\theta}w\para{.,0,.}=0\,\,\,\,\,\,\,\,\,\,\,\,\,\,\,\,\,\,\,\,\,\,\,\,\,\,\,\,\,\,\,\,\,\,\,\,\mbox{in}\,\,(0,T)\times\Omega'
  \\&\partial_{x_1}w\para{.,1,.}-e^{i\theta}\partial_{x_1}w\para{.,0,.}=0\,\,\,\,\,\,\,\,\,\,\,\,\,\,\,\,\,\,\,\mbox{in}\,\,(0,T)\times\Omega'.
  \end{array}
\right.
\end{equation}
Let $h\in \check{X}_\theta$. From Remark \ref{R1}, we know that there exists a unique solution $u_1\in \mathcal{C}^1\para{\cro{0,T};L^2\para{\check{\Omega}}}\cap \mathcal{C}\para{\cro{0,T};H^2\para{\check{\Omega}}}$ of the following BVP
\begin{equation}\label{86}
\left\{
  \begin{array}{lll}
    & \para{i\partial_t +\Delta_A}u_1=0\,\,\,\,\,\,\,\,\,\,\,\,\,\,\,\,\,\,\,\,\,\,\,\,\,\,\,\,\,\,\,\,\,\,\,\,\,\,\,\,\,\,\,\,\,\,\,\,\,\mbox{in}\,\,\check{\Omega},
    \\&u_1(T,.)=0\,\,\,\,\,\,\,\,\,\,\,\,\,\,\,\,\,\,\,\,\,\,\,\,\,\,\,\,\,\,\,\,\,\,\,\,\,\,\,\,\,\,\,\,\,\,\,\,\,\,\,\,\,\,\,\,\,\,\,\,\,\,\,\,\,\mbox{in}\,\,\check{\Omega}
    \\&u_1=h\,\,\,\,\,\,\,\,\,\,\,\,\,\,\,\,\,\,\,\,\,\,\,\,\,\,\,\,\,\,\,\,\,\,\,\,\,\,\,\,\,\,\,\,\,\,\,\,\,\,\,\,\,\,\,\,\,\,\,\,\,\,\,\,\,\,\,\,\,\,\,\,\,\,\,\,\,\,\mbox{in}\,\,
    \check{\Sigma}
    \\&u_1(.,1,.)=e^{i\theta}u_1(.,0,.)\,\,\,\,\,\,\,\,\,\,\,\,\,\,\,\,\,\,\,\,\,\,\,\,\,\,\,\,\,\,\,\,\,\,\,\,\mbox{in}\,\,(0,T)\times\Omega'
  ,\\&\partial_{ x_1}u_1(.,1,.)=e^{i\theta}\partial _{x_1}u_1(.,0,.)\,\,\,\,\,\,\,\,\,\,\,\,\,\,\,\,\,\,\,\mbox{in}\,\,(0,T)\times\Omega'.
  \end{array}
\right.
\end{equation}
In view of \eqref{87} and \eqref{86} we get by applying the Green formula that
\begin{equation}\label{88}
\begin{array}{lll}
\displaystyle\int_{\check{Q}}\para{i\partial_t+\Delta_{A}}w\overline{u_1}dx dt&=&\displaystyle\int_{\check{Q}}w\overline{\para{i\partial_t+\Delta_{A}}u_1}dxdt-\int_{\check{\Sigma}}\para{\partial_\nu+i A\cdot \nu}w\overline{u_1}d\sigma dt\\&=&\displaystyle-\int_{\check{\Sigma}}\para{\partial_\nu+i A\cdot \nu}w\overline{u_1}d\sigma dt.
\end{array}
\end{equation}
\subsection{Geometric optics solutions in periodic media}
Let $R>0$ be so large, such that $\overline{\Omega'}\subset B(0,R)$ and set
$$
\mathfrak{D}_R=B(0,R+1)\backslash \overline{B(0,R)}\subset \R^2,
$$
where  $B\para{a,R}$ denotes the ball in $\R^2$ centered at $a\in\R^2$ with radius $r>0.$
Let $\phi_0\in\mathcal{C}_0^\infty\left(\R^2\right)$ supported in $\mathfrak{D}_R$. Notice that
$$
\mbox{supp}\,\phi_0\cap \Omega'=\emptyset.
$$
Take $\sigma$ so large that $\sigma>\frac{2R+1}{T}$, in such a way that for all $\omega'\in\mathbb{S}^1$ we have
$$
\left(\mbox{supp}\,\phi_0\pm\sigma T \omega'\right)\cap \Omega'=\emptyset.
$$
Indeed, it is clear from the embedding $\overline{\Omega'}\subset B(0,R)$ and from the fact that $\mbox{supp}\,\phi_0\cap\Omega'=\emptyset$, that for each   $x'\in\mbox{supp}\,\phi_0\pm\sigma T\omega'$, we have
 $$x'=y'\pm\sigma T\omega',\,\,\,\,\,\,\, y'\in\mbox{supp}\,\phi_0$$
hence
 $$\abs{x'}\geq\sigma T-\abs{y'}>2R+1-\abs{y'}>R.$$
Since, $\mbox{supp}\left(\phi_0\right) \subset B\para{0,R+1},$ this entails that $x'\notin\Omega'.$ Next, for $\phi_{\theta}\in H^2_{\theta}\left(\R^3\right)$ and $\phi_0\in\mathcal{C}^\infty_0\para{\R^2}$ supported in $\mathfrak{D}_R,$
we put
$$
\phi\para{x_1,x'}=\phi_{\theta}\para{x_1,x'}\phi_0\para{x'},\,\,\para{x_1,x'}\in\R^3.
$$
Then it is apparent that the function $\Phi$ given by
\begin{equation}\label{phi}
\Phi\left(t,x\right)=\phi\left(x_1,x'-t\omega'\right),\,\,\,t\in\R,\,\para{x_1,x'}\in\R^3,
\end{equation}
is solution to the transport equation
$$
\left(\partial_t+\omega'\cdot\nabla_{x'}\right)\Phi\left(t,x\right)=0,\,\,\,\para{t,x}\in\R\times\R^3.
$$
Let $A=\left(a_1,a_2,a_3\right)\in W^{3,\infty}\left(\Omega,\R^3\right)$ fulfill $(\ref{(1.2)})$ and put
$$
A'=\left(a_2,a_3\right).
$$
We extend $A'$ by 0 outside $\Omega$ and we set
$$
b\left(t,x\right)=\exp\left(-i\int_0^t \omega'\cdot A'\left(x_1,x'-s\omega'\right)ds\right).
$$
We have for all $\para{t,x}\in \R\times\R^3$
$$
\begin{array}{lll}\displaystyle
\omega
'\cdot \nabla_{x'}b\para{t,x}&=&\displaystyle-ib\para{t,x}\int_0^t\sum_{k=2}^3 \omega_k\sum_{j=2}^3\omega_j\partial_ja_k\left(x_1,x'-s\omega'\right)ds\\&=&\displaystyle ib\para{t,x}\sum_{k=2}^3\omega_k\int_0^t\frac{d}{ds}a_k(x_1,x'-s\omega')ds\\&=&\displaystyle i\omega'\cdot A'\left(x_1,x'-t\omega'\right)b\para{t,x}-i\omega'\cdot A'\para{x_1,x'}b\para{t,x}\\&=&\displaystyle-\partial_tb\para{t,x}-i\omega'\cdot A'\para{x}b\para{t,x}.
\end{array}
$$
Therefore $b$ satisfies the equation
$$
\left(\partial_t+\omega'\cdot \nabla_{x'}+i\omega'\cdot A'\right)b=0.
$$
For $\omega'\in \mathbb{S}^1$, we consider the following subspace $\mathcal{H}^2_{\theta,\omega'}\para{\mathfrak{D}_R}$ of  $H^2\left(\para{0,1}\times\R^2\right)$ made of functions $\phi=\phi_\theta\phi_0\in H^2\para{\para{0,1}\times\R^2}$ such that $\omega'\cdot\nabla_{x'}\phi\in H^2\para{\para{0,1}\times\R^2},$ where $\phi_\theta\in H^2_\theta\para{\para{0,1}\times\R^2} $ and $\phi_0\in H^2\para{\R^2}$ satisfies $\mbox{supp}\left(\phi_0\right) \subset \mathfrak{D}_R$.
This space is equipped with its natural norm:
$$
\mathcal{N}_{\omega'}\left(\phi\right)=\left\|\phi\right\|_{H^2\left(\para{0,1}\times\mathbb{R}^2\right)}+\|\omega'\cdot\nabla_{x'}\phi\|_{H^2(\para{0,1}\times\mathbb{R}^2)}.
$$
In the rest of this paper we assume that $\sigma\geq1.$\\
Let us now prove the following lemma.
\begin{lemma}\label{1}
Fix $\omega'\in \mathbb{S}^1,\,\theta\in [0,2\pi)$ and let $A\in W^{3,\infty}\para{\Omega;\mathbb{R}^3}\cap \mathcal{A}$ satisfies \eqref{(1.2)}. If $\phi\in\mathcal{H}_{\theta,\omega'}^2\para{\mathfrak{D}_R}$ then the equation
\begin{equation}
\para{i\partial_t+\Delta_A}u=0,\,\,\,\,\para{t,x}\in \check{Q},
\end{equation}
admits a solution $u\in\mathcal{C}^1\para{\cro{0,T};L^2\para{\check{\Omega}}}\cap \mathcal{C}\para{\cro{0,T};H_\theta^2\para{\check{\Omega}}},$ of the form
$$
u\para{t,x}=\Phi\para{2\sigma t,x}b\para{2\sigma t,x} e^{i\sigma\para{x'\cdot\omega'-\sigma t}}+\psi_\sigma\para{t,x},
$$
where $\Phi$ is defined in \eqref{phi} and $\psi_\sigma$ satisfies
$$
\left\{
  \begin{array}{lll}
     &\hbox{$ \psi_\sigma\para{0,.}=0,\,\,\,\,\,\,\,\,\,\,\,\,\,\,\,\,\,\,\,\,\,\,\,\,\,\,\,\,\,\,\,\,\,\,\,\,\,\,\,\,\,\,\,\,\,\,\,\,\,\,\,\,\,\,\,\,\,\,\,\,\,\,\,\,\,\,\,\,\,\,\mbox{in}\,\,\check{\Omega} $},
  \\& \psi_\sigma=0\,\,\,\,\,\,\,\,\,\,\,\,\,\,\,\,\,\,\,\,\,\,\,\,\,\,\,\,\,\,\,\,\,\,\,\,\,\,\,\,\,\,\,\,\,\,\,\,\,\,\,\,\,\,\,\,\,\,\,\,\,\,\,\,\,\,\,\,\,\,\,\,\,\,\,\,\,\,\,\,\,\,\,\,\,\mbox{on}\,\,\check{\Sigma},\\&\psi_\sigma\para{.,1,.}=e^{i\theta}\psi_\sigma\para{.,0,.}\,\,\,\,\,\,\,\,\,\,\,\,\,\,\,\,\,\,\,\,\,\,\,\,\,\,\,\,\,\,\,\,\,\,\,\,\,\,\,\mbox{in}\,\,\para{0,T}\times\Omega',\\
  &\partial_{x_1}\psi_\sigma\para{.,1,.}=e^{i\theta}\partial_{x_1}\psi_\sigma\para{.,0,.}\,\,\,\,\,\,\,\,\,\,\,\,\,\,\,\,\,\,\,\,\,\,\mbox{in}\,\,\para{0,T}\times\Omega'.
  \end{array}
\right.
$$
Moreover, we have the estimate
\begin{equation}
\sigma\norm{\psi_\sigma}_{L^2\para{\check{Q}}}+\norm{\nabla\psi_\sigma}_{L^2\para{\check{Q}}}\leq C\mathcal{N}_{\omega'}\para{\phi},
\end{equation}
where $C$ depends only on $\check{\Omega},\,T$ and $\norm{A}_{W^{3,\infty}\para{\Omega}}$.
\end{lemma}
\begin{proof}
For notational simplicity, we set
$$
E_\sigma\para{t,x'}=e^{i\sigma\para{x'\cdot\omega'-\sigma t}}\,\mbox{and}\,\,\,\, \varphi_\sigma\para{t,x}= \Phi\para{2\sigma t,x}b\para{2\sigma t,x},
$$
so that we have
$$
u=\varphi_\sigma E_\sigma+ \psi_\sigma.
$$
Clearly, $\psi_\sigma$ must be a solution of the following IBVP
\begin{equation}\label{6}
\left\{
  \begin{array}{lll}
  &\para{i\partial_t+\Delta_A}\psi_\sigma=G\,\,\,\,\,\,\,\,\,\,\,\,\,\,\,\,\,\,\,\,\,\,\,\,\,\,\,\,\,\,\,\,\,\,\,\,\,\,\,\,\,\,\,\,\,\,\,\,\,\,\,\,\,\,\,\mbox{in}\,\,\check{Q},
   \\
     &\hbox{$ \psi_\sigma\para{0,.}=0\,\,\,\,\,\,\,\,\,\,\,\,\,\,\,\,\,\,\,\,\,\,\,\,\,\,\,\,\,\,\,\,\,\,\,\,\,\,\,\,\,\,\,\,\,\,\,\,\,\,\,\,\,\,\,\,\,\,\,\,\,\,\,\,\,\,\,\,\,\,\,\,\mbox{in}\,\,\check{\Omega} $},\\ & \psi_\sigma=0\,\,\,\,\,\,\,\,\,\,\,\,\,\,\,\,\,\,\,\,\,\,\,\,\,\,\,\,\,\,\,\,\,\,\,\,\,\,\,\,\,\,\,\,\,\,\,\,\,\,\,\,\,\,\,\,\,\,\,\,\,\,\,\,\,\,\,\,\,\,\,\,\,\,\,\,\,\,\,\,\,\,\,\,\,\mbox{on}\,\,\check{\Sigma},\\&\psi_\sigma\para{.,1,.}=e^{i\theta}\psi_\sigma\para{.,0,.}\,\,\,\,\,\,\,\,\,\,\,\,\,\,\,\,\,\,\,\,\,\,\,\,\,\,\,\,\,\,\,\,\,\,\,\,\,\,\,\mbox{in}\,\,\para{0,T}\times\Omega',\\
  &\partial_{x_1}\psi_\sigma\para{.,1,.}=e^{i\theta}\partial_{x_1}\psi_\sigma\para{.,0,.}\,\,\,\,\,\,\,\,\,\,\,\,\,\,\,\,\,\,\,\,\,\,\mbox{in}\,\,\para{0,T}\times\Omega',
  \end{array}
\right.
\end{equation}
where $G=-\para{i\partial_t+\Delta_A}\para{E_\sigma\varphi_\sigma}$.\\
Since $E_\sigma$ and $\varphi_\sigma$ are the respective solutions of the following two equations
$$
\para{i\partial_t+\Delta_A}E_\sigma=\para{-2\sigma\omega'\cdot A'-\abs{A}^2+i\mbox{div}A}E_\sigma,\,\,\,\,\mbox{in}\,\,\R\times\R^2.
$$
$$
\para{\partial_t+2\sigma\omega'\cdot \nabla_{x'}+2i\sigma\omega'\cdot A'}\varphi_\sigma=0,\,\,\,\,\mbox{in}\,\,\R\times\R^3,
$$
we obtain that  $G=-E_\sigma \Delta_A \varphi_\sigma$.\\We have $G\in H_0^1\para{0,T;L^2\para{\check{\Omega}}}\subset W^{1,1}\para{0,T;L^2\para{\check{\Omega}}}$, by our assumptions. Hence the IBVP $\para{\ref{6}}$ has a unique solution
$$
\psi_\sigma\in\mathcal{C}\para{[0,T];H_\theta^2(\check{\Omega})\cap L^2(0,1;H_0^1\para{\Omega'})} \cap \mathcal{C}^1\para{[0,T];L^2(\check{\Omega})},
$$
by Lemma \ref{lm-c}, and
$$
\begin{array}{lll}
\displaystyle\norm{\psi_\sigma}_{L^2\para{\check{Q}}}&\leq&\displaystyle C\int_0^T\norm{G_0\para{2\sigma t,.}}_{L^2\para{\check{\Omega}}}dt\leq \displaystyle\frac{C}{\sigma} \norm{\phi}_{H^2\para{\para{0,1}\times\mathbb{R}^2}},
\end{array}
$$
where $G_0=\Delta_A\varphi_\sigma$.\\ \\
Moreover, in view of  Lemma  \ref{lm-c}, we have for any $\epsilon>0$
$$
\begin{array}{lll}
\norm{\nabla\psi_\sigma}_{L^2\para{\check{Q}}}&\leq &\displaystyle C\epsilon\int_0^T\para{\sigma^2\norm{G_0\para{2\sigma t,.}}_{L^2\para{\check{\Omega}}}+\sigma\norm{\partial_t G_0\para{2\sigma t,.}}_{L^2\para{\check{\Omega}}}}dt\\&\,&\displaystyle+\epsilon^{-1}\int_0^T \norm{G_0\para{2\sigma t,.}}_{L^2\para{\check{\Omega}}} dt .
\end{array}
$$
Choosing $\epsilon=\sigma^{-1}$ in the above identity we obtain
$$
\begin{array}{lll}
\norm{\nabla\psi_\sigma\para{t}}_{L^2\para{\check{Q}}}&\leq &\displaystyle C\para{\int_\mathbb{R}\norm{G_0\para{s,.}}_{L^2\para{\check{\Omega}}}ds+\int_{\mathbb{R}}\norm{\partial_t G_0\para{s,.}}_{L^2\para{\check{\Omega}}}ds}\\&\leq&\displaystyle C\para{\norm{\phi}_{H^2\para{\para{0,1}\times\mathbb{R}^2}}+\norm{\omega'\cdot\nabla_{x'}\phi}_{H^2\para{\para{0,1}\times\mathbb{R}^2}}},
\end{array}
$$
which completes the proof.
\begin{remark}\label{R3}
We have a similar result by replacing above the condition $\psi_\sigma\para{0,.}=0\,\,\mbox{in}\,\,\check{\Omega}$ by the condition $\psi_\sigma\para{T,.}=0\,\,\mbox{in}\,\,\check{\Omega}.$
\end{remark}
\end{proof}
\section{Stability estimate}
\label{sec-si}
\subsection{Preliminary estimate}
Let $\omega'\in \mathbb{S}^1$ and let $A_j\in W^{3,\infty}\para{\Omega;\mathbb{R}^3}\cap\mathcal{A}$ satisfies \eqref{(1.2)} and $\norm{A_j}_{ W^{3,\infty}}\leq M,$ for $j=1,2$. We extend $A_j$ by zero outside $\Omega$ and still denote by $A_j$ the resulting function. We set $A=\para{a_1,a_2,a_3}=\para{a_1,A'}=A_2-A_1.$ and for $\omega'\in\mathbb{S}$,
$$
b_j\para{t,x}=\exp\para{-i\int_0^t\omega'\cdot A'_j\para{x_1,x'-s\omega'}ds}.
$$
We introduce
$$\displaystyle b\para{t,x}=\para{b_2\overline{b}_1}\para{t,x}=\exp\para{-i\int_0^t\omega'\cdot A'\para{x_1,x'-s\omega'}ds},\,\,\para{t,x}\in\R\times\R^3,$$
where
\begin{equation}\label{5.6}
A'\para{x}=
\left\{
  \begin{array}{lll}&A'_2-A'_1\,\,\,\,\mbox{if $x\in\Omega$}\\& 0\,\,\,\,\,\,\,\,\,\,\,\,\,\,\,\,\,\,\,\,\,\,\,\mbox{if $x\in \R^3\backslash\Omega$}

  \end{array}
\right.
\end{equation}
Since $A'_1-A'_2=0$ on $\partial\Omega$  then $A'=\para{a_2,a_3}$ belongs to $H^1\para{\R^3}$ and we may regard $\frac{\partial a_2}{\partial x_3}-\frac{\partial a_3}{\partial x_2}$ as a function in $L^2\para{\R^3}$ which is supported in $\Omega$.
\begin{lemma}\label{0.1}
We assume that  $\sigma>2R/T$. Then there exists a constant  $C=C\para{M,\Omega'}>0$ such that for any $\omega'\in \mathbb{S}^{1}$, and all $\phi_1,\phi_2\in\mathcal{H}_{\theta,\omega'}^2\para{\mathfrak{D}_R}$, the following estimate holds
\begin{equation}\label{70}
\begin{array}{lll}
\displaystyle \abs{\int_0^T\int_0^1\int_{\mathbb{R}^2}\sigma\omega'\cdot A'\para{x}\para{\phi_2\overline{\phi}_1}\para{x_1,x'-2\sigma t\omega'}b\para{2\sigma t,x}dxdt}\leq\\\,\,\,\,\,\,\,\,\,\,\,\,\,\,\,\,\,\,\,\,\,\,\,\,\,\,\,\,\,\,\,\,\,\,\,\,\,\,\,\,\,\,\,\,\,\,\,\,\,\,\,\,\,\,\,\,\,\,\,\,\,\,\,\,\,\,\,\,\,\,\,\,\,\,\,\,\,\,\,\,\,\,\,\,\,\,\,\,\,\,\,\,\,\,\,\,\,\,\,\,\,\,\,\,\,\,\,\,\,\, C\para{\sigma^2\norm{\Lambda_{A_1,\theta}-\Lambda_{A_2,\theta}}+\sigma^{-1}}\mathcal{N}_{\omega'}\para{\phi_1}\mathcal{N}_{\omega'}\para{\phi_2}.
\end{array}
\end{equation}
\end{lemma}
\begin{proof}
Let $u_2\in\mathcal{C}^1\para{\cro{0,T};L^2\para{\check{\Omega}}}\cap \mathcal{C}\para{\cro{0,T};H^2_\theta\para{\check{\Omega}}}$ be the geometric optics solution of
$$
\para{i\partial_t+\Delta_{A_2}}u=0,\,\,\,\,\para{t,x}\in \check{Q},
$$
given by Lemma $\ref{1}$, with expression
\begin{equation}\label{71}
u_2\para{t,x}=\Phi_2\para{2\sigma t,x}b_2\para{2\sigma t,x} e^{i\sigma\para{x'\cdot\omega'-\sigma t}}+\psi_{2,\sigma}\para{t,x},
\end{equation}
where $\psi_{2,\sigma}$ satisfies
\begin{equation}
\left\{
  \begin{array}{lll}\label{5.1}
     &\psi_{2,\sigma}\para{0,.}=0\,\,\,\,\,\,\,\,\,\,\,\,\,\,\,\,\,\,\,\,\,\,\,\,\,\,\,\,\,\,\,\,\,\,\,\,\,\,\,\,\,\,\,\,\,\,\,\,\,\,\,\,\,\,\,\,\,\,\,\,\,\,\,\,\,\,\,\,\,\,\,\,\,\,\,\,\mbox{in}\,\,\check{\Omega},\\

     &\psi_{2,\sigma}=0\,\,\,\,\,\,\,\,\,\,\,\,\,\,\,\,\,\,\,\,\,\,\,\,\,\,\,\,\,\,\,\,\,\,\,\,\,\,\,\,\,\,\,\,\,\,\,\,\,\,\,\,\,\,\,\,\,\,\,\,\,\,\,\,\,\,\,\,\,\,\,\,\,\,\,\,\,\,\,\,\,\,\,\,\,\,\,\,\mbox{on}\,\,\check{\Sigma},\\
     &\psi_{2,\sigma}\para{.,1,.}=e^{i\theta}\psi_{2,\sigma}\para{.,0,.}\,\,\,\,\,\,\,\,\,\,\,\,\,\,\,\,\,\,\,\,\,\,\,\,\,\,\,\,\,\,\,\,\,\,\,\,\,\,\,\mbox{in}\,\,\para{0,T}\times\Omega',\\
  &\partial_{x_1}\psi_{2,\sigma}\para{.,1,.}=e^{i\theta}\partial_{x_1}\psi_{2,\sigma}\para{.,0,.}\,\,\,\,\,\,\,\,\,\,\,\,\,\,\,\,\,\,\,\,\,\,\mbox{in}\,\,\para{0,T}\times\Omega',
  \end{array}
\right.
\end{equation}
and
\begin{equation}\label{72}
 \sigma\norm{\psi_{2,\sigma}}_{L^2\para{\check{Q}}}+\norm{\nabla\psi_{2,\sigma}}_{L^2\para{\check{Q}}}\leq C\mathcal{N}_{\omega'}\para{\phi_2}.
 \end{equation}
Put
$$
f_{\sigma,2}\para{t,x}=\Phi_2\para{2\sigma t,x}b_2\para{2\sigma t,x} e^{i\sigma\para{x'\cdot\omega'-\sigma t}},\,\,\,\,x\in\check{\Sigma},\,\,\,t\in\para{0,T},
$$
so that we have $f_{\sigma,2}\in \check{X}_\theta$ and $f_{\sigma,2}=u_2$ on $\check{\Sigma}$. Let $v$ by the $\mathcal{C}^1\para{\cro{0,T};L^2\para{\check{\Omega}}}\cap \mathcal{C}\para{\cro{0,T};H^2_\theta\para{\check{\Omega}}}$ solution of the following IBVP
\begin{equation}\label{7}
\left\{
  \begin{array}{lll}
    & \para{i\partial_t+\Delta_{A_1}}v=0\,\,\,\,\,\,\,\,\,\,\,\,\,\,\,\,\,\,\,\,\,\,\,\,\,\,\,\,\,\,\,\,\,\,\,\,\,\,\,\,\,\,\,\,\,\,\,\,\mbox{in}\,\,\check{Q},\\ &\hbox{$ v\para{0,.}=0\,\,\,\,\,\,\,\,\,\,\,\,\,\,\,\,\,\,\,\,\,\,\,\,\,\,\,\,\,\,\,\,\,\,\,\,\,\,\,\,\,\,\,\,\,\,\,\,\,\,\,\,\,\,\,\,\,\,\,\,\,\,\,\,\,\,\mbox{in}\,\,\check{\Omega} $},\\
     &v=f_{\sigma,2}\,\,\,\,\,\,\,\,\,\,\,\,\,\,\,\,\,\,\,\,\,\,\,\,\,\,\,\,\,\,\,\,\,\,\,\,\,\,\,\,\,\,\,\,\,\,\,\,\,\,\,\,\,\,\,\,\,\,\,\,\,\,\,\,\,\,\,\,\,\,\,\,\mbox{on}\,\,\check{\Sigma},\\
     &\hbox{$ v\para{.,1,.}=e^{i\theta}v\para{.,0,.}\,\,\,\,\,\,\,\,\,\,\,\,\,\,\,\,\,\,\,\,\,\,\,\,\,\,\,\,\,\,\,\,\,\,\,\mbox{in}\,\,\check{\Omega} $},\\ &\hbox{$ \partial_{x_1}v\para{.,1,.}=e^{i\theta}\partial_{x_1}v\para{.,0,.}\,\,\,\,\,\,\,\,\,\,\,\,\,\,\,\,\,\,\mbox{in}\,\,\check{\Omega} $},
  \end{array}
\right.
\end{equation}
given by Remark \ref{R2}. In view of \eqref{5.1} and \eqref{7}, the function $w=v-u_2\in\mathcal{C}\para{\cro{0,T};H^2_{\theta}\para{\check{\Omega}}}\cap L^2\para{0,T;  H_0^1\para{\check{\Omega}}}\cap\mathcal{C}^1\para{\cro{0,T};L^2\para{\check{\Omega}}},
$ satisfies the IBVP
$$
\left\{
  \begin{array}{lll}
    & \para{i\partial_t+\Delta_{A_1}}w=2iA\cdot\nabla u_2+Vu_2\,\,\,\,\,\,\,\,\,\,\,\,\,\,\,\,\,\,\,\,\,\,\,\,\,\,\,\,\,\,\,\,\,\,\,\,\,\,\,\,\,\,\,\,\,\,\,\,\,\,\,\,\,\,\,\,\,\mbox{in}\,\,\check{Q},\\
     &\hbox{$ w\para{0,.}=0\,\,\,\,\,\,\,\,\,\,\,\,\,\,\,\,\,\,\,\,\,\,\,\,\,\,\,\,\,\,\,\,\,\,\,\,\,\,\,\,\,\,\,\,\,\,\,\,\,\,\,\,\,\,\,\,\,\,\,\,\,\,\,\,\,\,\,\,\,\,\,\,\,\,\,\,\,\,\,\,\,\,\,\,\,\,\,\,\,\,\,\,\,\,\,\,\,\,\,\,\,\,\,\,\,\,\,\,\,\,\,\,\,\,\,\,\mbox{in}\,\,\check{\Omega} $},\\
     &w=0\,\,\,\,\,\,\,\,\,\,\,\,\,\,\,\,\,\,\,\,\,\,\,\,\,\,\,\,\,\,\,\,\,\,\,\,\,\,\,\,\,\,\,\,\,\,\,\,\,\,\,\,\,\,\,\,\,\,\,\,\,\,\,\,\,\,\,\,\,\,\,\,\,\,\,\,\,\,\,\,\,\,\,\,\,\,\,\,\,\,\,\,\,\,\,\,\,\,\,\,\,\,\,\,\,\,\,\,\,\,\,\,\,\,\,\,\,\,\,\,\,\,\,\,\,\,\,\,\,\mbox{on}\,\,\check{\Sigma},
     \\
     &\hbox{$ w\para{.,1,.}=e^{i\theta}v\para{.,0,.}\,\,\,\,\,\,\,\,\,\,\,\,\,\,\,\,\,\,\,\,\,\,\,\,\,\,\,\,\,\,\,\,\,\,\,\,\,\,\,\,\,\,\,\,\,\,\,\,\,\,\,\,\,\,\,\,\,\,\,\,\,\,\,\,\,\,\,\,\,\,\,\,\,\,\,\,\,\,\,\,\,\,\,\,\,\,\mbox{in}\,\,\check{\Omega} $},\\ &\hbox{$ \partial_{x_1}w\para{.,1,.}=e^{i\theta}\partial_{x_1}w\para{.,0,.}\,\,\,\,\,\,\,\,\,\,\,\,\,\,\,\,\,\,\,\,\,\,\,\,\,\,\,\,\,\,\,\,\,\,\,\,\,\,\,\,\,\,\,\,\,\,\,\,\,\,\,\,\,\,\,\,\,\,\,\,\,\,\,\,\,\,\,\,\mbox{in}\,\,\check{\Omega} $},
  \end{array}
\right.
$$
with
$$
V=i \,\mbox{div}\para{A} -\abs{A_2}^2+\abs{A_1}^2.
$$
With reference to Remark \ref{R3}, let
\begin{equation}\label{5.7}
u_1\para{t,x}=\Phi_1\para{2\sigma t,x}b_1\para{2\sigma t,x} e^{i\sigma\para{x'\cdot\omega'-\sigma t}}+\psi_{1,\sigma}\para{t,x},
\end{equation}
be a $\mathcal{C}^1\para{\cro{0,T};L^2\para{\check{\Omega}}}\cap \mathcal{C}\para{\cro{0,T};H^2_\theta\para{\check{\Omega}}}$  solution of the equation
$$
\para{i\partial_t+\Delta_{A_1}}u=0,\,\,\,\,\mbox{in}\,\,\check{Q},
$$
where
$\psi_{1,\sigma}$ satisfies
\begin{equation}\label{8}
\left\{
  \begin{array}{lll}
  &\psi_{1,\sigma}\para{T,.}=0\,\,\,\,\,\,\,\,\,\,\,\,\,\,\,\,\,\,\,\,\,\,\,\,\,\,\,\,\,\,\,\,\,\,\,\,\,\,\,\,\,\,\,\,\,\,\,\,\,\,\,\,\,\,\,\,\,\,\,\,\,\,\,\,\,\,\,\,\,\,\,\,\,\,\,\,\mbox{in}\,\,\check{\Omega},
     \\&\psi_{1,\sigma}=0\,\,\,\,\,\,\,\,\,\,\,\,\,\,\,\,\,\,\,\,\,\,\,\,\,\,\,\,\,\,\,\,\,\,\,\,\,\,\,\,\,\,\,\,\,\,\,\,\,\,\,\,\,\,\,\,\,\,\,\,\,\,\,\,\,\,\,\,\,\,\,\,\,\,\,\,\,\,\,\,\,\,\,\,\,\,\,\,\,\mbox{on}\,\,\check{\Sigma},
     \\&\psi_{1,\sigma}\para{.,1,.}=e^{i\theta}\psi_{1,\sigma}\para{.,0,.}\,\,\,\,\,\,\,\,\,\,\,\,\,\,\,\,\,\,\,\,\,\,\,\,\,\,\,\,\,\,\,\,\,\,\,\,\,\,\,\mbox{in}\,\,\para{0,T}\times\Omega',\\
  &\partial_{x_1}\psi_{1,\sigma}\para{.,1,.}=e^{i\theta}\partial_{x_1}\psi_{1,\sigma}\para{.,0,.}\,\,\,\,\,\,\,\,\,\,\,\,\,\,\,\,\,\,\,\,\,\,\mbox{in}\,\,\para{0,T}\times\Omega'.
  \end{array}
\right.
\end{equation}
and
\begin{equation}\label{5.8}
\sigma\norm{\psi_{1,\sigma}}_{L^2\para{\check{Q}}}+\norm{\nabla\psi_{1,\sigma}}_{L^2\para{\check{Q}}}\leq C\mathcal{N}_{\omega'}\para{\phi_1}.
\end{equation}
It follows from identity $\para{\ref{88}}$ that
\begin{equation}\label{cc}
\begin{array}{lll}
\displaystyle\int_{\check{Q}}\para{i\partial_t+\Delta_{A_1}}w\overline{u_1}dx dt&=&\displaystyle\int_{\check{Q}}2iA\cdot \nabla u_2 \overline{u_1}dx dt+\int_{\check{Q}}V\para{x} u_2  \overline{u_1}dx dt\\&=&\displaystyle-\int_{\check{\Sigma}}\para{\partial_\nu+i A_1\cdot \nu}w\overline{u_1}d\sigma dt.
\end{array}
\end{equation}
Since $A=0$ on $\partial \Omega,$ and hence on $\check{\Sigma}$ by \eqref{5.6}, then $\para{\ref{7}}$ and $\para{\ref{cc}}$ give
\begin{equation}\label{5.2}
\begin{array}{lll}
\displaystyle\int_{\check{Q}} 2i A\cdot\nabla u_2 \overline{u_1}dx dt+\int_{\check{Q}} V\para{x} u_2 \overline{u_1}dx dt&=&\displaystyle-\int_{\check{\Sigma}} \para{\Lambda_{A_1,\theta}-\Lambda_{A_2,\theta}}\para{f_{\sigma,2}}\overline{f_{\sigma,1}} d\sigma dt\\&=&\displaystyle -\langle\para{\Lambda_{A_1,\theta}-\Lambda_{A_2,\theta}}\para{f_{\sigma,2}},\overline{f_{\sigma,1}}\rangle,
\end{array}
\end{equation}
where we have set
$$
f_{\sigma,1}\para{t,x}=\Phi_1\para{2\sigma t,x} b_1\para{2\sigma t,x} e^{i\sigma\para{x'\cdot\omega'-\sigma t}},\,\,\,\,\,\,\para{t,x}\in\check{\Sigma}.
$$
On the other hand by \eqref{71} and \eqref{5.7} we have
\begin{equation}\label{5.3}
\displaystyle\int_{\check{Q}} 2iA\cdot \nabla u_2\overline{u_1}dx dt=\displaystyle -\int_{\check{Q}} 2\sigma \omega' \cdot A'\para{x}\para{\phi_2 \overline{\phi}_1}\para{x_1,x'-2\sigma t\omega'}\para{b_2 \overline{b}_1}\para{2\sigma t,x} dx dt+\mathcal{I}_{\sigma},
\end{equation}
where
$$
\begin{array}{lll}
 \mathcal{I}_{\sigma} &=&\displaystyle\int_{\check{Q}}2iA\cdot \nabla\para{\Phi_2\para{2\sigma t,x}b_2\para{2\sigma t,x}}\overline{\Phi}_1 \para{2\sigma t,x} \overline{b}_1\para{2\sigma t,x} dx dt \\&\,&+ \displaystyle\int_{\check{Q}} 2iA\cdot \nabla\para{\Phi_2\para{2\sigma t,x}b_2\para{2\sigma t,x}}e^{i\sigma\para{x'\cdot\omega'-\sigma t}}\overline{\psi}_{1,\sigma}\para{t,x}dx dt\\&\,&+ \displaystyle\int_{\check{Q}}2iA\cdot \nabla\psi_{2,\sigma} \para{2\sigma t,x} \overline{\Phi}_1\para{2\sigma t,x} \overline{b}_1\para{2\sigma t,x} e^{i\sigma\para{x'\cdot\omega'-\sigma t}}\\&\,&\displaystyle +\int_{\check{Q}}2iA\cdot\nabla \psi_{2,\sigma}\para{t,x} \overline{\psi}_{1,\sigma}\para{t,x}dx dt\\&\,&-\displaystyle \int_{\check{Q}} 2\sigma \omega'\cdot A'\para{x}b_2\para{2\sigma t,x} \Phi_2\para{2\sigma t,x} \overline{\psi}_{1,\sigma} e^{i\sigma\para{x'\cdot\omega'-\sigma t}} dx dt
 \end{array}
$$
Using \eqref{72} and \eqref{5.8}, we obtain  that
\begin{equation}\label{5.5}
\abs{\mathcal{I}_{\sigma}}\leq C\sigma^{-1}\mathcal{N}_{\omega'}\para{\phi_2}\mathcal{N}_{\omega'}\para{\phi_1}.
\end{equation}
From this and \eqref{5.2}-\eqref{5.3}, it follows that
\begin{eqnarray}\label{9.1}
&&\abs{\displaystyle \int_{\check{Q}}\sigma\omega'\cdot A'\para{x}\para{\phi_2\overline{\phi}_1}\para{x_1,x'-2\sigma t\omega'}\para{b_2  \overline{b}_1}\para{2\sigma t,x}dxdt}\cr&&\leq\displaystyle C \para{\abs{\int_{\check{Q}} V u_2\overline{u}_1dxdt}+\displaystyle\abs{\int_{\Sigma_1}\para{\Lambda_{A_1,\,\theta}-\Lambda_{A_2,\,\theta}}\para{f_{\sigma,2}}\overline{f_{\sigma,1}}d\sigma dt} +\sigma^{-1}\mathcal{N}_{\omega'}\para{\phi_2}\mathcal{N}_{\omega'}\para{\phi_1}}.
\end{eqnarray}
On the other  hand \eqref{71}, \eqref{72}, \eqref{5.7} and \eqref{5.8} imply that
\begin{equation}\label{9.2}
\abs{\int_{\check{Q}} V\para{x}u_2\overline{u}_1 dx dt}\leq C\sigma^{-2}\mathcal{N}_{\omega'}\para{\phi_2}\mathcal{N}_{\omega'}\para{\phi_1}\leq C\sigma^{-1}\mathcal{N}_{\omega'}\para{\phi_2}\mathcal{N}_{\omega'}\para{\phi_1},
\end{equation}
and we have
\begin{eqnarray}\label{9.3}
\abs{\int_{\check{\Sigma}}\para{\Lambda_{A_1,\theta}-\Lambda_{A_2,\theta}}\para{f_{\sigma,2}}\overline{f_{\sigma,1}}d\sigma dt}&=&\displaystyle \abs{\langle\para{\Lambda_{A_1,\theta}-\Lambda_{A_2,\theta}}\para{f_{\sigma,2}},f_{\sigma,1}\rangle_{L^2\para{{\check{\Sigma}}}}}\cr&\leq&\norm{(\Lambda_{A_1,\theta}-\Lambda_{A_2,\theta})(f_{\sigma,2})}_{L^2\para{{\check{\Sigma}}}}\norm{f_{\sigma,1}}_{L^2\para{{\check{\Sigma}}}}\cr&\leq&\norm{\Lambda_{A_1,\theta}-\Lambda_{A_2,\theta}}\norm{f_{\sigma,2}}_{\check{X}_\theta}\norm{f_{\sigma,1}}_{L^2\para{{\check{\Sigma}}}}\cr&\leq&C\sigma^{2}\mathcal{N}_{\omega'}\para{\phi_2}\mathcal{N}_{\omega'}\para{\phi_1}\norm{\Lambda_{A_1,\theta}-\Lambda_{A_2,\theta}}.
\end{eqnarray}
From   \eqref{9.1}-\eqref{9.2} and \eqref{9.3}, we derive that
\begin{eqnarray}
&&\abs{\int_0^T\int_0^1\int_{\mathbb{R}^2}\sigma\omega'\cdot A'\para{x}\para{\phi_2 \overline{\phi}_1}\para{x_1,x'-2\sigma t\omega'}b\para{2\sigma t,x}dx dt}\cr&&\leq C\para{\sigma^{2}\mathcal{N}_{\omega'}\para{\phi_1}\mathcal{N}_{\omega'}\para{\phi_2}\norm{\Lambda_{A_1,\theta}-\Lambda_{A_2,\theta}} +\sigma^{-1}\mathcal{N}_{\omega'}\para{\phi_1}\mathcal{N}_{\omega'}\para{\phi_2}}.
\end{eqnarray}
This completes the proof of the lemma.
\end{proof}
\subsection{Two technical results}
Let $f$ be a function defined on $\R^3$ and let $\omega'\in \mathbb{S}^1$. We define the X-ray transform of $f$ at $x$ in the direction $\omega'$ by
$$
(\mathcal{P}f)\para{\omega',x}=\int_\mathbb{R} f\para{x_1,x'+s\omega'}ds,\,\,x=\para{x_1,x'}\in \mathbb{R}^3.
$$
We observe that $\para{\mathcal{P}f}\para{\omega',x}$ where $x=\para{x_1,x'}$ does not change if $x'$ is moved in the direction of $\omega'$. Therefore we restrict  $x'$ to $
 \omega'^{\perp}=\{\kappa \in \mathbb{R}^2;\,\,\kappa\cdot\omega'=0\}$. Hereafter, $\widehat{f}$ denotes the partial Fourier transform of the function $f$ with respect to the cross-section variable $x'\in\Omega',$ i.e
$$
\widehat{f}\para{x_1,\xi'}=\para{2\pi}^{-1}\int_{\R^2} f\para{x_1,x'} e^{-ix'\cdot\xi'} dx',\,\,\,\,\xi'\in \mathbb{R}^2,\,x_1\in\R.
$$
\begin{lemma}\label{0.33}
Let $f\in L^1\para{\mathbb{R}^3}$ and $\omega'\in \mathbb{S}^1$. Then  $(\mathcal{P}f)\para{\omega',.}\in L^1\para{\mathbb{R}\times{\omega'}^{\perp}}$ and
$$
\widehat{\para{(\mathcal{P}f)\para{\omega',.}}}\para{x_1,\xi'}=\para{2\pi}^{-1}\int_{\omega'^{\perp}} e^{-ix'\cdot \xi'} (\mathcal{P}f)\para{\omega',x_1,x'}dx'=\widehat{f}\para{x_1,\xi'},
$$
for all $\xi'\in {\omega'}^{\perp}$.
\end{lemma}
\begin{proof}
Obviously
$$
\int_{\R\times{\omega'}^{\perp}} \abs{(\mathcal{P}f)\para{\omega',x}}dx \leq \int_{\R\times{\omega'}^{\perp}}\int_{\R} \abs{f\para{x_1,x'+t\omega'}}dtdx_1dx'=\int_{\R^3} \abs{f\para{x}}dx<\infty.
$$
The change of variable $y'=x'+t\omega'\in{\omega'}^{\perp}\oplus \mathbb{R}\omega'$, yields  $dy'=dx'dt $ and after noting that $\xi'\in{\omega'}^{\perp}$ implies $x'\cdot \xi'=x'\cdot \xi'+t\omega'\cdot\xi'=y'\cdot \xi'$,
$$
\begin{array}{lll}
\widehat{\para{(\mathcal{P}f)\para{\omega',.}}}\para{x_1,\xi'}&=&\displaystyle\para{2\pi}^{-1}\int_{\omega'^{\perp}} \int_{\R}f\para{x_1,x'+t\omega'}e^{-ix'\cdot \xi'}dtdx'\\&=& \displaystyle\para{2\pi}^{-1}\int_{\R\times\omega'^{\perp}}f\para{x_1,x'+t\omega'}e^{-ix'\cdot \xi'}dtdx'\\&=&\displaystyle\para{2\pi}^{-1}\int_{\mathbb{R}^2} f\para{x_1,y'}e^{-iy'\cdot \xi'}dy'=\widehat{f}\para{x_1,\xi'}.
\end{array}
$$
\end{proof}
For $j=1,2,3$ we introduce the notation
\begin{equation}\label{6.2}
\rho_j\para{x}=\omega'\cdot \frac{\partial A'}{\partial  x_j}\para{x}=\sum_{i=2}^3 \omega_i  \frac{\partial a_i}{\partial  x_j}\para{x},\,\,\,x\in \mathbb{R}^3.
\end{equation}
Next for $\omega'=\para{\omega_2,\omega_3}\in \mathbb{S}^1$, we put
$$
\mathfrak{D}_R^-\para{\omega'}=\{x'\in\mathfrak{D}_R,\,x'\cdot \omega' <0\}.
$$
Let us now prove the following Lemma.
\begin{lemma}\label{6.4}
  Put $\sigma_0=2(R+1)/T.$ then there exists a constant $C=C\para{A_1,M}>0$ such that for all  $\omega'\in \mathbb{S}^1$ and all $\phi=\phi_{\theta}\phi_0\in \mathcal{H}_{\omega',\theta}^2\para{\mathfrak{D}_R}$ satisfying  $\mbox{supp}\para{\phi_0}\subset \mathfrak{D}_R^-$ and $\partial_j\phi\in\mathcal{H}_{\omega',\theta}^2\para{\mathfrak{D}_R}$ for $j\in\{2,3\}$, the following estimate
$$
\abs{\int_0^1e^{-i2k\pi x_1}\int_{\R^2}\phi^2\para{x}(\mathcal{P}\rho_j)\para{\omega',x}\exp\para{-i\int_{\mathbb{R}}\omega'\cdot A'\para{x_1,x'+s\omega'}ds}dx} \leq
$$
$$
\,\,\,\,\,\,\,\,\,\,\,\,\,\,\,\,\,\,\,\,\,\,\,\,\,\,\,\,\,\,\,\,\,\,\,\,\,\,\,\,\,\,\,\,\,\,\,\,\,\,\,\,\,\,\,\,\,\,\,\,\,\,\,\,\,\,\,\,\,\,\,\,\,\,\,\,\,\,\,\,\,\,\,\,\,\,\,\,\,\,\,\,\,\,\,\,\,\,\,\,\,\,\,\,\,\,\,\,\,\,\,\,\,\,\,\,\,\,\,\,\,\,\,\,\,\,\,\,C\para{\sigma^2\norm{\Lambda_{A_1,\theta}-\Lambda_{A_2,\theta}}+\frac{1}{\sigma}}\mathcal{N}_{\omega'}\para{\phi}\mathcal{N}_{\omega'}\para{\partial_j\phi},
$$
holds for any $\sigma>\sigma_0,\,k\in\mathbb{Z}$ and $j=2,3.$
\end{lemma}
\begin{proof}
For $\phi_1,\phi_2\in \mathcal{H}_{\omega',\theta}^2\para{\mathfrak{D}_R}$, we have
\begin{eqnarray}\label{6.3}
&&\int_0^T\int_0^1\int_{\mathbb{R}^2}\sigma \omega'\cdot A'\para{x}\para{\phi_2 \overline{\phi}_1}\para{x_1,x'-2\sigma t\omega' }b\para{2\sigma t, x} dx'dx_1dt \cr&=&\int_0^T\int_0^1\int_{\mathbb{R}^2}\sigma \omega'\cdot A'\para{x_1,x'+2\sigma t\omega'}\para{\phi_2 \overline{\phi}_1}\para{x}b\para{2\sigma t,x_1, x'+2\sigma t\omega'}dx'dx_1dt\cr&=&\displaystyle\int_0^1\int_{\mathbb{R}^2}\para{\phi_2 \overline{\phi}_1}\para{x}\int_0^T\sigma \omega'\cdot A'\para{x_1,x'+2\sigma t\omega'} \exp\para{-i\int_0^{2\sigma t} \omega'\cdot A'\para{x_1,x'+s\omega'}ds}dtdx'dx_1 \cr&=&\displaystyle\frac{i}{2}\int_0^1\int_{\mathbb{R}^2}\para{\phi_2 \overline{\phi}_1}\para{x}\int_0^T \frac{d}{dt} \exp\para{-i\int_0^{2\sigma t} \omega'\cdot A'\para{x_1,x'+s\omega'}ds}dtdx_1dx'\cr&=&\displaystyle\frac{i}{2}\int_0^1\int_{\mathbb{R}^2}\para{\phi_2 \overline{\phi}_1}\para{x}\cro{ \exp\para{-i\int_0^{2\sigma T} \omega'\cdot A'\para{x_1,x'+s\omega'}ds}-1}dx.
\end{eqnarray}
 We choose $\phi_1$ and $\phi_2$ such that $\phi_2\para{x}=e^{-i2k\pi x_1}\phi\para{x}$, $\phi_1=\partial_j \overline{\phi},\,j\in\{2,3\}$. By an application of Green's formula, $\para{\ref{6.3}}$ yields
\begin{eqnarray}\label{5.9}
\displaystyle&&\int_0^T\int_0^1\int_{\mathbb{R}^2}\sigma \omega'\cdot A'\para{x}\para{\phi_2 \overline{\phi}_1}\para{x_1,x'-2\sigma t\omega' }b\para{2\sigma t, x} dx dt \displaystyle \cr&=& -\frac{i}{4}\int_0^1\int_{\mathbb{R}^2}e^{-i2k\pi x_1} \phi^2\para{x}\frac{\partial }{\partial x_j}\cro{ \exp\para{-i\int_0^{2\sigma T} \omega'\cdot A'\para{x_1,x'+s\omega'}ds}}dx\cr&=& -\frac{1}{4}\int_0^1\int_{\mathbb{R}^2}e^{-i2k\pi x_1}\phi^2\para{x}\frac{\partial }{\partial x_j}\para{\int_0^{2\sigma T} \omega'\cdot A'\para{x_1,x'+s\omega'}ds}\cr&\,&\,\,\,\,\,\,\,\,\,\,\,\,\,\,\,\,\,\,\,\,\,\,\,\,\,\,\,\,\,\,\,\,\,\,\,\,\,\,\,\,\,\,\,\,\,\,\,\,\,\,\,\,\,\,\,\,\,\,\,\,\,\,\,\,\,\,\,\,\,\,\,\,\,\,\,\,\,\,\,\,\,\,\,\,\,\,\,\,\,\,\,\,\,\,\,\,\,\,\,\,\,\,\,\,\,\,\,\,\,\,\,\,\,\,\,\,\,\,\,\,\,\,\,\,\,\,\,\,\,\,\,\,\,\,\,\,\,\,\,\,\,\,\, \exp\para{-i\int_0^{2\sigma T} \omega'\cdot A'\para{x_1,x'+s\omega'}ds}dx.
\end{eqnarray}
Since the support of $A'$ is contained in $\R\times B\para{0,R}$ and $2\sigma T>R$ we have
\begin{equation}\label{6.1}
\int_0^{2\sigma T} \omega'\cdot A'\para{x_1,x'+s\omega'}ds=\int_\mathbb{R} \omega'\cdot A'\para{x_1,x'+s\omega'}ds,
\end{equation}
for all $x'\in\mathfrak{D}_R^-\para{\omega}\,$. In fact, for all $s\geq 2\sigma T$ and $x'\in\mathfrak{D}_{R}$ it is easy to see that
$\para{x_1,x'+s\omega'} \notin \mbox{supp}\para{A'},$ for each $x_1\in [0,1]$. Therefore we have
\begin{equation}\label{73}
\int_0^{2\sigma T} \omega'\cdot A'\para{x_1,x'+s\omega'}ds=\int_0^\infty \omega'\cdot A'\para{x_1,x'+s\omega'}ds,\,\,\para{x_1,x'}\in [0,1]\times\mathfrak{D}_R^-\para{\omega'}.
\end{equation}
On the other hand, if $s\leq 0$ and $x'\in\mathfrak{D}_R^-$ it holds true that $\abs{x'+s\omega'}^2=\abs{x'}^2+s^2+2sx'\cdot \omega'\geq R^2$ hence $A\para{x_1,x'+s\omega'}=0$. This and \eqref{73} entail $\para{\ref{6.1}}$. Further, upon inserting $\para{\ref{6.1}}$ into the equation \eqref{5.9}, we obtain
$$
\begin{array}{lll}
\displaystyle\int_0^T\int_0^1\int_{\mathbb{R}^2}\sigma \omega'\cdot A'\para{x}\para{\phi_2 \overline{\phi}_1}\para{x_1,x'-2\sigma t\omega'}b\para{2\sigma t, x}dxdt\\= \displaystyle-\frac{1}{4}\int_0^1\int_{\mathbb{R}^2}e^{-i2k\pi x_1}\phi^2\para{x}\frac{\partial }{\partial x_j}\para{\int_{\mathbb{R}}\omega'\cdot A'\para{x_1,x'+s\omega'}ds} \exp\para{-i\int_{\mathbb{R}}\omega'\cdot A'\para{x_1,x'+s\omega'}ds}dx\\=\displaystyle-\frac{1}{4}\int_0^1\int_{\mathbb{R}^2} e^{-i2k\pi x_1}\phi^2\para{x}\mathcal{P}\para{\rho_j}\para{\omega',x} \exp\para{-i\int_{\mathbb{R}} \omega'\cdot A'\para{x_1,x'+s\omega'}ds}dx,
\end{array}
$$
where $\rho_j$ is given by $\para{\ref{6.2}}$. From this and Lemma $\ref{0.1}$, we obtain for any $\sigma>\sigma_0$ that
$$
\abs{\int_0^1e^{-i2k\pi x_1}\int_{\mathbb{R}^2}\phi^2\para{x}\mathcal{P}\para{\rho_j}\para{\omega',x}\exp\para{-i\int_{\mathbb{R}}\omega'\cdot A'\para{x_1,x'+s\omega'}ds}dx'dx_1} \leq
$$
$$
\,\,\,\,\,\,\,\,\,\,\,\,\,\,\,\,\,\,\,\,\,\,\,\,\,\,\,\,\,\,\,\,\,\,\,\,\,\,\,\,\,\,\,\,\,\,\,\,\,\,\,\,\,\,\,\,\,\,\,\,\,\,\,\,\,\,\,\,\,\,\,\,\,\,\,\,\,\,\,\,\,\,\,\,\,\,\,\,\,\,\,\,\,\,\,\,\,\,\,\,\,\,\,\,\,\,\,\,\,\,\,\,\,\,\,\,\,\,\,\,\,\,\,\,\,\,\,\,C\para{\sigma^2\norm{\Lambda_{A_1,\theta}-\Lambda_{A_2,\theta}}+\frac{1}{\sigma}}\mathcal{N}_{\omega'}\para{\phi}\mathcal{N}_{\omega'}\para{\partial_j\phi}.
$$
The proof is then complete.
\end{proof}
\subsection{Estimate for the magnetic potential}
Let us now prove the estimate of Theorem \ref{thmm}.
We start by estimating the Fourier transform of the function $\beta_{23}$ where
$$
\beta_{ij}=\frac{\partial a_i}{\partial x_j}-\frac{\partial a_j}{\partial x_i},\,\,\,\,i,j=1,2,3.
$$
\begin{lemma}\label{lm.2}
 Let $\sigma_0=2R/T$. Then there exists a constant  $C=C\para{A_1,M}$ such that for any $\sigma > \sigma_0$ the following estimate
\begin{equation}\label{0.3}
\abs{\int_0^1e^{i2k\pi x_1}\widehat{\beta}_{23}\para{x_1,\xi'}dx_1}\leq C\para{\sigma^2\norm{\Lambda_{A_1,\theta }-\Lambda_{A_2,\theta}}+\frac{1}{\sigma}}\langle \para{k,\xi'}\rangle^5,
\end{equation}
holds uniformly in $k\in \mathbb{Z}$ and $\xi'\in\R^2$. Here $\langle \para{k,\xi'}\rangle=\para{1+k^2+\abs{\xi'}^2}^{1/2}$ and $\widehat{\beta}_{23}$ denotes the partial Fourier transform of $\beta_{23}$ with respect to $x'$.
\end{lemma}
\begin{proof}{}
We fix $z'_0\in \omega'^{\bot}\cap B\para{0,R+1/2}.$ Let $h\in\mathcal{C}_0^\infty \para{\R}$ be supported in $\para{0,1/8}$ and satisfy the condition
$$
\int_{\mathbb{R}} h^2\para{t}dt=1.
$$
Let
$$
r_{z_0'}=\sqrt{\para{R+\frac{3}{4}}^2-\abs{z_0'}},\,\,\,\,\,z'_1=z'_0-r_{z_0'}\omega'.
$$
It is not difficult to check that
$$
B\para{z_1',1/4}\subset \mathfrak{D}_R^-\para{\omega'}.
$$
Let $\beta_0\in \mathcal{C}_0^\infty \para{ \omega'^{\bot}\cap B\para{z'_0,1/8}}$ be nonnegative and for $y=\para{y_1,y'}\in \mathbb{R}^3$, put
$$
\phi_\theta\para{y}=e^{i\theta y_1}\exp\para{\frac{i}{2}\int_\mathbb{R}\omega'\cdot A'\para{y_1,y'+s\omega'}ds},\,\,y=\para{y_1,y'}\in\mathbb{R}^3,
$$
and
$$
\phi_0\para{y'}=h\para{y'\cdot \omega'+r_{z'_0}} e^{-\frac{i}{2}y'\cdot \xi'}\beta_0^{1/2}\para{y'-\para{y'\cdot \omega'}\omega'},\,\,y'\in\mathbb{R}^2.
$$
It is apparent that
$$
\mbox{supp}\para{\phi_0}\subset B\para{z'_1, 1/4} \subset \mathfrak{D}_R^-\para{\omega'}.
$$
Set
\begin{equation}\label{0.0}
\phi\para{y}=\phi_\theta\para{y}\phi_0\para{y'},\,\,y=\para{y_1,y'}\in\mathbb{R}^3.
\end{equation}
It is clear that $\phi\in\mathcal{H}_{\omega',\theta}^2\para{\mathfrak{D}_R}.$
By performing the change of variable $y'=x'+t\omega' \in \omega'^\perp \oplus \R\omega'$ in the following integral, we get upon recalling that $\xi' \in \omega'^\perp$, that
$$
\begin{array}{lll}
&&\displaystyle\int_0^1\int_{\mathbb{R}^2} e^{-i2k\pi x_1}\phi^2\para{x}\para{\mathcal{P}\rho_j}\para{\omega',x} \exp\para{-i\int_{\mathbb{R}} \omega'\cdot A'\para{x_1,x'+s\omega'}ds}dx\\&=&\displaystyle\int_0^1\int_{\mathbb{R}}\int_{\omega'^\perp}e^{-i2k\pi x_1}\phi^2\para{x_1,x'+t\omega'}\para{\mathcal{P}\rho_j}\para{\omega,x_1,x'+t\omega'} \\&\,& \displaystyle \,\,\,\,\,\,\,\,\,\,\,\,\,\,\,\,\,\,\,\,\,\,\,\,\,\,\,\,\,\,\,\,\,\,\,\,\,\,\,\,\,\,\,\,\,\,\,\,\,\,\,\,\,\,\,\,\,\,\,\,\,\,\,\,\,\,\,\,\,\,\,\,\,\,\,\,\,\,\,\,\,\,\,\,\,\,\,\,\,\,\,\,\,\,\,\,\,\,\,\,\,\,\,\,\,\,\,\,\,\,\,\,\,\,\,\,\,\,\,\, \exp\para{-i\int_{\mathbb{R}} \omega'\cdot A'\para{x_1,x'+s\omega'}ds}dx'dtdx_1\\&=&\displaystyle\int_0^1e^{-i2k\pi x_1}\int_{\mathbb{R}}\int_{\omega'^\perp}h^2\para{t+r_{z_0'}}e^{-ix'\cdot \xi'}\beta_0\para{x'}\para{\mathcal{P}\rho_j}\para{\omega',x_1,x'}dx'dtdx_1
\\&=&\displaystyle\int_0^1e^{-i2k\pi x_1}\int_{\omega'^\perp}e^{-ix'\cdot \xi'}\beta_0\para{x'}\para{\mathcal{P}\rho_j}\para{\omega',x_1,x'}dx'dx_1.
\end{array}
$$
It follows from this and Lemma \ref{6.4} that
$$
\begin{array}{lll}
\displaystyle\abs{\int_0^1e^{-i2k\pi x_1}\int_{\omega'^\perp}e^{-ix'\cdot \xi'}\beta_0\para{x'}\para{\mathcal{P}\rho_j}\para{\omega,x_1,x'}dx'dx_1} \\ \displaystyle \,\,\,\,\,\,\,\, \,\,\,\,\,\,\,\,\,\,\,\,\,\,\,\,\,\,\,\,\,\,\,\,\,\,\,\,\,\,\,\,\,\,\,\,\,\,\,\,\,\,\,\,\,\,\,\,\,\,\,\,\,\,\,\,\,\,\,\,\,\,\,\,\,\,\,\,\,\,\,\,\,\,\,\,\,\,\,\,\,\,\,\,\,\,\,\,\,\,\,\,\,\,\,\,\,\,\,\,\,\,\,\, \leq C\para{\sigma^2\norm{\Lambda_{A_1,\theta}-\Lambda_{A_2,\theta}}+\frac{1}{\sigma}}\mathcal{N}_{\omega'}\para{\phi}\mathcal{N}_{\omega'}\para{\partial_j\phi}.
\end{array}
$$
As $\phi$ is given by $\para{\ref{0.0}}$ and $\mathcal{N}_{\omega'}\para{\phi}=\norm{\phi}_{H^2\para{(0,1)\times \mathbb{R}^2}}+\norm{\omega'\cdot\nabla_{x'}\phi}_{H^2\para{(0,1)\times \mathbb{R}^2}}$, an elementary calculation gives for any $\xi'\in\omega'^\perp$
$$
\mathcal{N}_{\omega'}\para{\phi}\mathcal{N}_{\omega'}\para{\partial_j\phi}\leq C \langle (k,\xi') \rangle^5,
$$
where $C>0$ is independent of $k$ and $\xi'$.\\
From the last two inequalities we derive for all $\xi'\in \omega'^\perp$ and $k\in\Z$ that
\begin{equation}
\begin{array}{lll}
\displaystyle\abs{\int_0^1e^{-i2k\pi x_1}\int_{\omega'^\perp}e^{-ix'\cdot \xi'}\para{\mathcal{P}\rho_j}\para{\omega',x}dx'dx_1 } \leq C \para{\sigma^2\norm{\Lambda_{A_1,\theta}-\Lambda_{A_2,\theta}}+\frac{1}{\sigma}}\langle (k,\xi') \rangle^5.
\end{array}
\end{equation}
By Lemma \ref{0.33} we have
$$
\abs{\int_0^1e^{-i2k\pi x_1}\widehat{\rho_j}\para{x_1,\xi'}dx_1} \leq C \para{\sigma^2\norm{\Lambda_{A_1,\theta}-\Lambda_{A_2,\theta}}+\frac{1}{\sigma}}\langle (k,\xi') \rangle^5.
$$
In view of the identity $\xi'\cdot\omega'=0$ we have for $j=2,3,$
$$
\widehat{\rho}_j\para{.,\xi'}=\sum_{i=2}^3 \omega_i \xi_j \widehat{a}_i\para{.,\xi'}=\sum_{i=2}^3 \omega_i \para{\xi_j \widehat{a}_i\para{.,\xi'}-\xi_i \widehat{a}_j\para{.,\xi'}}=\sum_{i=2}^3 \omega_i\hat{\beta}_{ij}\para{.,\xi'}.
$$
Since $\omega'\in\mathbb{S}^1$ is arbitrary, we get,  for any $ \xi'\in \R^2$ and $k\in\Z$ that
$$
\abs{\int_0^1e^{-i2k\pi x_1}\widehat{\beta_{23}}\para{x_1,\xi'}dx_1} \leq C \para{\sigma^2\norm{\Lambda_{A_1,\theta}-\Lambda_{A_2,\theta}}+\frac{1}{\sigma}}\langle (k,\xi') \rangle^5,
$$
proving the result.
\end{proof}
Having established  Lemma \ref{lm.2} we may now terminate the proof of Theorem $\ref{thmm}$.
For simplicity, we use the following notation
$$
\phi_k\para{x_1}=e^{-i2k\pi x_1},\,\,k\in\Z,
$$
and
$$
\widehat{\mathfrak{b}}\para{\xi',k}=\langle \widehat{\beta_{23}}\para{\xi'}, \phi_k\rangle_{L^2\para{0,1}}=\int_0^1e^{-i2k\pi x_1}\widehat{\beta_{23}}\para{x_1,\xi'}dx_1.
$$
Then, by the Parseval-Plancherel theorem, we find that
\begin{equation}\label{eq1}
\norm{\beta_{23}}^2_{L^2\para{\check{\Omega}}}=\sum_{k\in\mathbb{Z}}\int_{\mathbb{R}^2}\abs{\widehat{\mathfrak{b}}\para{\xi',k}}^2d\xi'=\int_{\mathbb{R}^3}\abs{\widehat{\mathfrak{b}}\para{\xi',k}}^2d\xi'd\mu\para{k},
\end{equation}
 where $\mu=\sum_{n\in\mathbb{Z}}\delta_n$. For $\gamma>0$, put $\displaystyle B_\gamma=\{\para{\xi',k}\in\mathbb{R}^2\times\mathbb{Z},\,\langle\para{\xi',k}\rangle\leq \gamma\}$. We shall treat $\int_{B_\gamma}\abs{\widehat{\mathfrak{b}}\para{\xi',k}}^2d\xi'd\mu\para{k}$ and $\int_{\mathbb{R}^3\backslash B_\gamma}\abs{\widehat{\mathfrak{b}}\para{\xi',k}}^2d\xi'd\mu\para{k},$ separately. First, it holds true that
$$
\begin{array}{lll}
\displaystyle\int_{\mathbb{R}^3\backslash B_\gamma}\abs{\widehat{\mathfrak{b}}\para{\xi',k}}^2d\xi'd\mu\para{k}&\leq&\displaystyle \frac{1}{\gamma} \int_{\mathbb{R}^3\backslash B_\gamma}\abs{\para{\xi',k}}^2\abs{\widehat{\mathfrak{b}}\para{\xi',k}}^2d\xi'd\mu\para{k} \\&\leq&\displaystyle \frac{1}{\gamma^2} \int_{\mathbb{R}^3\backslash B_\gamma}\para{1+\abs{\para{\xi',k}}^2}\abs{\widehat{\mathfrak{b}}\para{\xi',k}}^2d\xi'd\mu\para{k} \\&\leq& \displaystyle  \frac{C}{\gamma^2}\norm{\beta_{23}}^2_{H^1\para{\Omega}}\\&\leq& \displaystyle  \frac{C M^2}{\gamma^2}.
\end{array}
$$
Further, in light of  \eqref{0.3} we have
$$
\begin{array}{lll}
\displaystyle\int_{ B_\gamma}\abs{\widehat{\mathfrak{b}}\para{\xi',k}}^2d\xi'd\mu\para{k} &\leq&\displaystyle C\para{\sigma^2 \norm{\Lambda_{A_1,\theta}-\Lambda_{A_2,\theta}}+\frac{1}{\sigma}}^2\int_{ B_\gamma}\langle\para{\xi',k}\rangle^{10} d\xi' d\mu\para{k}\\&\leq&\displaystyle C\para{\sigma^2 \norm{\Lambda_{A_1,\theta}-\Lambda_{A_2,\theta}}+\frac{1}{\sigma}}^2\gamma^{13},
\end{array}
$$
so we deduce for \eqref{eq1} that
$$
\begin{array}{lll}
\displaystyle\norm{\beta_{23}}_{L^2\para{\check{\Omega}}}^2&=& \displaystyle\int_{B_\gamma}\abs{\widehat{\mathfrak{b}}\para{\xi',k}}^2d\xi'd\mu\para{k}+ \int_{\mathbb{R}^3\backslash B_\gamma}\abs{\widehat{\mathfrak{b}}\para{\xi',k}}^2d\xi'd\mu\para{k}\\&\leq& \displaystyle C\para{\sigma^4 \gamma^{13} \norm{\Lambda_{A_1,\theta}-\Lambda_{A_2,\theta}}^2+\frac{\gamma^{13}}{\sigma^2}+\frac{1}{\gamma^2}}.
\end{array}
$$
Thus, choosing
\begin{equation}\label{0.2}
\sigma^2=\gamma^{15},
\end{equation}
we obtain, that
\begin{equation}\label{0.5}
\norm{\beta_{23}}_{L^2\para{\check{\Omega}}}\leq C\para{\gamma^{43/2} \norm{\Lambda_{A_1,\theta}-\Lambda_{A_2,\theta}}+\frac{1}{\gamma}}.
\end{equation}
The above reasoning being valid for $\sigma\geq \sigma_0.$ According to $\para{\ref{0.2}}$, we need to take $\gamma\geq \gamma_0$ where $\gamma_0=\sigma_0^{2/15}$. Therefore, for $\norm{\Lambda_{A_1,\theta}-\Lambda_{A_2,\theta}}\leq \gamma_0^{-45/2}$ and $\gamma=\norm{\Lambda_{A_1,\theta}-\Lambda_{A_2,\theta}}^{-2/45}\geq \gamma_0,$ we have
$$
\norm{\beta_{23}}_{L^2\para{\check{\Omega}}}\leq C\norm{\Lambda_{A_1,\theta}-\Lambda_{A_2,\theta}}^{2/45}.
$$
Now if  $\norm{\Lambda_{A_1}-\Lambda_{A_2}}\geq \gamma_0^{-45/2},$ it holds true that
\begin{equation}\label{0.4}
\norm{\beta_{23}}_{L^2\para{\check{\Omega}}}\leq \chi\para{\Omega'} \norm{\beta_{23}}_{L^\infty\para{\check{\Omega}}}\leq C\norm{A}_{W^{3,\infty}\para{\check{\Omega}}}\leq \frac{CM}{\gamma_0}\gamma_0\leq \frac{C}{\gamma_0}\norm{\Lambda_{A_1,\theta}-\Lambda_{A_2,\theta}}^{2/45}.
\end{equation}
Thus it follows from $\para{\ref{0.5}}$ and $\para{\ref{0.4}}$ that,
$$
\norm{\beta_{23}}_{L^2\para{\check{\Omega}}}\leq C\norm{\Lambda_{A_1,\theta}-\Lambda_{A_2,\theta}}^{2/45}.
$$
This ends the proof of Theorem $\ref{thmm}$.


\bigskip


\end{document}